\documentclass[3p,times]{elsarticle}

\usepackage{fontenc}
\usepackage{amssymb}
\usepackage{lipsum}
\usepackage{amsfonts}
\usepackage{graphicx}
\usepackage{epstopdf}
\usepackage{algorithmic}
\usepackage{latexsym}
\usepackage{color}
\usepackage{caption}
\usepackage{multirow}
\usepackage{setspace}
\usepackage{appendix}
\usepackage{float}
\usepackage{amsmath}
\usepackage{amscd}
\usepackage{mathrsfs}
\usepackage{amssymb}
\usepackage{multirow}
\usepackage{threeparttable}
\usepackage{booktabs}
\usepackage{rotating}
\usepackage{subeqnarray}
\usepackage{cases}
\usepackage{bbm}
\usepackage{amsthm}
\usepackage{txfonts}
\usepackage{bm}
\usepackage{tabularx}

\def\cL{{\cal L}}

\def\bbE{{\mathbb{E}}}

\def\be{\begin{eqnarray}}
\def\ee{\end{eqnarray}}
\def\ba{\begin{array}}
\def\ea{\end{array}}

\allowdisplaybreaks[4]

\begin{document}

\allowdisplaybreaks[4]

\newtheorem{lemma}{Lemma}[section]
\newtheorem{remark}{Remark}[section]
\newtheorem{example}{Example}[section]
\newtheorem{theorem}{Theorem}[section]
\newtheorem{corollary}{Corollary}[section]
\newtheorem{conjecture}{Conjecture}[section]
\newtheorem{definition}{Definition}[section]
\newtheorem{proposition}{Proposition}[section]
\newtheorem{condition}{Condition}[section]

\numberwithin{equation}{section}

\begin{frontmatter}

\title{Local discontinuous Galerkin method for nonlinear BSPDEs of Neumann boundary conditions with deep backward dynamic programming time-marching}

\author{
Yixiang Dai\tnoteref{label1}\quad\quad 
Yunzhang Li\tnoteref{label2}\quad\quad 
Jing Zhang\tnoteref{label1}}
\tnotetext[label1]{School of Mathematical Sciences, Fudan University, Shanghai 200433, China. { E-mail:} {yixiangdai@hotmail.com;  zhang\_jing@fudan.edu.cn}. 
The research of the third author is supported by the National Key R\&D Program of China (2022YFA1006101, 2018YFA0703903), and by the NSFC (12271103, 12031009) and by the Shanghai Science and Technology Commission Grant (21ZR1408600).}
\tnotetext[label2]{Corresponding Author. Research Institute of Intelligent Complex Systems, Fudan University, Shanghai 200433, China. { E-mail:} {li\_yunzhang@fudan.edu.cn}. Research supported by the National Natural Science Foundation of China (No.~12301566), and by the Science and Technology Commission of Shanghai Municipality (No.~23JC1400300), and by the Chenguang Program of Shanghai Education Development Foundation and Shanghai Municipal Education Commission (No.~22CGA01).} 



\begin{abstract}
This paper aims to present a local discontinuous Galerkin (LDG) method for solving backward stochastic partial differential equations (BSPDEs) with Neumann boundary conditions. We establish the $L^2$-stability and optimal error estimates of the proposed numerical scheme. Two numerical examples are provided to demonstrate the performance of the LDG method, where we incorporate a deep learning algorithm to address the challenge of the curse of dimensionality in backward stochastic differential equations (BSDEs). The results show the effectiveness and accuracy of the LDG method in tackling BSPDEs with Neumann boundary conditions.
\end{abstract}

\begin{keyword}
local discontinuous Galerkin method, backward stochastic partial differential equations, Neumann boundary problems, stability analysis, error estimates, deep learning algorithm

\vspace{0.5cm}
\MSC 65M60 \sep 65M15 \sep 65Z05

\end{keyword}

\end{frontmatter}

\section{Introduction}
\thispagestyle{empty}
Let $(\Omega,\bar{\mathscr{F}},\{\bar{\mathscr{F}}_t\}_{t\geq0},\mathbb{P})$ be a complete filtered probability space satisfying the usual conditions. The filtration $\{\bar{\mathscr{F}}_t\}_{t\geq0}$ is generated by two independent $d-$dimensional Wiener processes $W$ and $B$. We denote by $\{\mathscr{F}_t\}_{t\geq0}$ the natural filtration generated by $W$, together with all $\mathbb{P}-$null sets. The terminal time $T$ is a fixed positive number. The predictable $\sigma-$algebra on $\Omega\times[0,T]$ associated with $\{\mathscr{F}_t\}_{t\geq0}$ and $\{\bar{\mathscr{F}}_t\}_{t\geq0}$ is denoted by $\mathscr{P}$ and $\bar{\mathscr{P}}$, respectively. 

In this paper we present a local discontinuous Galerkin (LDG) method for the following backward stochastic partial differential equations (BSPDEs) with Neumann boundary conditions:
\begin{equation}\label{BSPDE_in_Qiu_paper}
\left\{
\begin{aligned}
       &-d u(x,t) = 
    \left[ \frac{1}{2} \left( |\sigma|^2 
    + |\bar{\sigma}|^2 \right) u_{xx}
    + \sigma \psi_x
    + \Gamma(\cdot, u, u_x, \psi) \right](x,t) dt 
     - \psi(x,t) dW_t, \quad\quad (x,t) \in [0, b]\times [0, T),\\
    & u_x(0,t) = g(0,t),\quad u_x(b,t) = g(b,t),\quad\quad t\in[0,T],\\
    &u(x,T) = G(x), \quad\quad  x \in [0, b], 
    \end{aligned}
    \right.
\end{equation}
which is associated to the following stochastic control problem: 
\begin{equation}\label{target_function_for_stochastic_control}
    \mathop{\text{min}}_{\theta} \mathbb{E} 
    \left[ \int_0^T f(t,X_t,\theta_t) dt 
    + \int_0^T g(t,X_t) dL_t 
    + \int_0^T g(t,X_t) dU_t + G(X_T) \right],
\end{equation}
subject to 
\begin{equation}\label{equation_for_state_process}
    \left\{
    \begin{aligned}
    &dX_t = \beta(t,X_t, \theta_t) dt 
    + \sigma(t,X_t) dW_t 
    + \bar{\sigma}(t,X_t) dB_t + dL_t - dU_t,
     \quad t \in (0, T], \\
    &X_0 = x,\quad L_0 = U_0 = 0, \\
    &0 \leq X_t \leq b, \quad\text{a.s.}, \\
    &\int_0^T X_t dL_t = \int_0^T (b - X_t) dU_t = 0, \quad\text{a.s.}, 
    \end{aligned}
    \right. 
\end{equation}
where $L$ and $U$ are two increasing processes. The above optimal control problem \eqref{target_function_for_stochastic_control}-\eqref{equation_for_state_process} has been studied by Bayraktar and Qiu in \cite{10.1214/19-AAP1465}, they proved the existence and uniqueness of a sufficiently regular solution for \eqref{BSPDE_in_Qiu_paper} which is then used to construct the optimal feedback control.

Backward stochastic partial differential equations (BSPDEs), as the infinite-dimensional version of backward stochastic differential equations (BSDEs), have been extensively used in problems related to probability theory and stochastic processes, for instance, in the optimal control problems of partial information stochastic differential equations (SDEs) or of parabolic stochastic partial differential equations (SPDEs), it appears as the dual equation of the Duncan–Mortensen–Zakai filtration equation (see for example \cite{Bensoussan83, Tang1998, XYZ93}), used to construct the stochastic maximum principle. Stochastic Hamilton-Jacobi-Bellman equations are a special type of nonlinear BSPDEs, first proposed by Peng in \cite{Peng92SHJB} to study the non-Markovian stochastic control problems. The relation between forward backward stochastic differential equations (FBSDEs) with random coefficients and BSPDEs which can be viewed as the stochastic Feynman–Kac formula has been established, see for example \cite{HuMaYong}. 


The linear, semilinear even if the quasilinear BSPDEs have been extensively studied; we refer to \cite{DuTang12, HuMaYong, MaYongPTRF, QiuTang, XYZ93} among many others.
However the existing works mainly focus on the BSPDEs in the whole space and the Dirichlet boundary problems, and few on the Neumann problem. For special forms of BSPDEs, the Neumann problem can be obtained through the semigroup method (see \cite{HuPengBSEE, Tessitore96}). In Bayraktar and Qiu \cite{10.1214/19-AAP1465}, they proved the existence and uniqueness of strong solution to the Neumann boundray problem for genenral nolinear BSPDEs.


In most cases, the solutions to BSPDEs are difficult to obtain explicitly, which has led to significant interest in numerical methods for solving BSPDEs. Based on the classical finite element methods, Yang and Zhao~\cite{yang2020finite} first proposed a finite element numerical method for BSPDEs with Dirichlet boundary conditions and provided a proof of second-order convergence. Later, Sun and Zhao~\cite{sun2024generalized} used finite element methods for spatial discretization and a $\theta$-scheme for temporal discretization to develop a finite element $\theta$-algorithm for numerically solving BSPDEs with Dirichlet boundary conditions, proving the convergence of the algorithm with respect to both time and space. However, due to the regularity requirements of basis functions, classical continuous finite element methods typically employ piecewise linear functions, which constrain the achievable accuracy. When high precision is needed, this often necessitates very fine meshes and extended computation times.

Inspired by Bassi and Rebay's use of the discontinuous Galerkin (DG) method for numerically solving compressible Navier-Stokes equations in~\cite{bassi1997high}, the local discontinuous Galerkin (LDG) method was first introduced by Cockburn and Shu in~\cite{cockburn1998local} for solving deterministic convection-diffusion equations. The LDG method extends the DG method, which was developed by Cockburn et al. in~\cite{cockburn1989tvb,cockburn1991runge,cockburn1998runge}. The LDG method is particularly suited for solving nonlinear hyperbolic systems and is characterized by its ease in designing high-order accurate algorithms, adaptability to different geometries, and suitability for efficient parallel computing. 

By employing the LDG method, with high precision the general BSPDEs can be converted into solving high-dimensional BSDEs. Li~\cite{li2022high} presented an LDG method for linear BSPDEs and demonstrated high-order convergence under sufficient regularity of the solution to the original equation, based on the fact that there is an explicit formulation for the solution to linear BSDE. However, for nonlinear case, the computational complexity often grows at least polynomially with the reciprocal of the approximation accuracy and the dimension of the solution space~\cite{novak2008tractability, bellman1966dynamic, novak1997curse}, which is referred to as the curse of dimensionality. To overcome this problem, E and Han et al.~\cite{E_2017, Han_2018} proposed a class of deep BSDE methods based on neural networks for solving high-dimensional partial differential equations (PDEs). The core idea is to transform PDEs into equivalent forward-backward stochastic differential equation systems (FBSDEs) and use neural networks for their numerical solution. These methods address the curse of dimensionality encountered in high-dimensional PDEs but exhibit some instability in the algorithms themselves. Huré and Germain et al.~\cite{hure2020deep, germain2022approximation} introduced a deep backward dynamic programming approach, which applies a backward discrete scheme to derive single-step dynamic programming principles, splits the training task, and iterates to solve the original problem, thus improving algorithm stability. Molla and Qiu~\cite{molla2023numerical} performed numerical solutions for a class of coupled forward-backward stochastic partial differential equations (FBSPDEs), first discretizing space with finite element methods to obtain equivalent FBSDEs, and then applying the deep BSDE algorithm of~\cite{hure2020deep} for numerical solution. This motivates us to use the deep BSDE method of ~\cite{hure2020deep} to numerically solve the BSDEs obtained through the LDG method’s discretization.

The structure of the remaining sections is as follows: Section 2 provides a review of the background on Neumann boundary problems for BSPDEs, introduces the requisite preliminaries for the LDG method, and establishes the notation used throughout the paper. Section 3 applies the LDG method to a class of nonlinear BSPDEs with Neumann boundary conditions, detailing the main results, which include the existence and uniqueness of the LDG numerical solution, the stability of the LDG method, and optimal error estimates. Section 4 presents the proofs of the main theorems. Finally, Section 5 outlines a framework for a class of deep backward dynamic programming algorithms and demonstrates the efficacy of the LDG method through numerical examples.

\section{Preliminaries}
\subsection{Notations and definition of solutions to BSPDEs}
The norm of a $d_1\times d_2$ matrix $M$ is given by $\left|M\right| := \sqrt{\text{trace}
\left(M^{T}M\right)}$. 
For a Banach space $V$, the space $L^2(\Omega,\mathscr{F}_T;V)$ is the set of all $V-$valued $\mathscr{F}_T-$measurable and square integrable random variables. For $p\in[1,\infty)$, we denote by $\mathcal{S}^p_{\mathscr{F}}(0,T;V)$ the set of all $V-$valued and $\{\mathscr{F}_t\}_{t\geq0}$-adapted c\`adl\`ag processes $(u_t)_{t\in[0,T]}$ such that 
$$\|u\|^p_{\mathcal{S}^p_{\mathscr{F}}(0,T;V)}=\bbE\sup_{t\in[0,T]}\|u_t\|^p_V<\infty.$$
By $\cL^p_{\mathscr{F}}(0,T;V)$ we denote the set of $V-$valued $\{\mathscr{F}_t\}_{t\geq0}$-adapted processes $(u_t)_{t\in[0,T]}$ such that
\begin{equation*}\begin{split}
&\|u\|^p_{\cL^p_{\mathscr{F}}(0,T;V)}:=\bbE\int_0^T\|u_t\|^p_Vdt<\infty,\quad p\in[1,\infty);
\\&\|u\|_{\cL^{\infty}_{\mathscr{F}}(0,T;V)}:={\mbox{ess}\sup}_{(\omega,t)\in\Omega\times[0,T]}\|u_t\|_V<\infty,\quad p=\infty. 
\end{split}\end{equation*}
For $k\in\mathbb{N}^+$ and $p\in[1,\infty)$, $H^{k,p}\left([0,b]\right)$ is the Sobolev space of all real-valued functions $u$ whose up-to $k$th order derivatives belong to $L^p([0,b])$, equipped with the usual Sobolev norm $\|u\|_{H^{k,p}}^{p}$.
By $H^{k,p}_{0}\left([0,b]\right)$, we denote the space of all the trace-zero functions in $H^{k,p}\left([0,b]\right)$. For $k=0$, $H^{0,p}\left([0,b]\right)
:=L^{p}\left([0,b]\right)$. We use $\|\cdot\|$ and $(\cdot,\cdot)$ to denote the norm and the inner product in the usual Hilbert space $L^2([0,b])$ and if there is no confusion, we shall also use $\langle\cdot,\cdot\rangle$ to denote the duality between Hilbert space $H^{k,2}([0,b])$ and its dual spaces.

Throughout this paper, we set for $k\in\mathbb{N}^+$,
$$\mathcal{H}:= \mathcal{S}^2_{\mathscr{F}}(0,T;L^2([0,b]))\cap\mathcal{L}^2_{\mathscr{F}}(0,T;H^{1,2}([0,b]))  \times\mathcal{L}^2_{\mathscr{F}}(0,T;L^2([0,b]))$$
and
$$\mathcal{H}^k:= \mathcal{S}^2_{\mathscr{F}}(0,T;H^{k,2}([0,b]))\cap\mathcal{L}^2_{\mathscr{F}}(0,T;H^{k+1,2}([0,b])) \times\mathcal{L}^2_{\mathscr{F}}(0,T;H^{k,2}([0,b])),$$
and they are complete spaces equipped with the norms 
$$\|(u,\psi)\|^2_{\mathcal{H}}:=\|u\|^2_{\mathcal{S}^2_{\mathscr{F}}(0,T;L^2([0,b]))}+\|u\|^2_{\mathcal{L}^2_{\mathscr{F}}(0,T;H^{1,2}([0,b]))}+\|\psi\|^2_{\mathcal{L}^2_{\mathscr{F}}(0,T;L^2([0,b]))}, $$
and
$$\|(u,\psi)\|^2_{\mathcal{H}^k}:=\|u\|^2_{\mathcal{S}^2_{\mathscr{F}}(0,T;H^{k,2}([0,b]))}+\|u\|^2_{\mathcal{L}^2_{\mathscr{F}}(0,T;H^{k+1,2}([0,b]))}+\|\psi\|^2_{\mathcal{L}^2_{\mathscr{F}}(0,T;H^{k,2}([0,b]))},$$
respectively.

Now we recall the assumptions and the existence and uniqueness result for \eqref{BSPDE_in_Qiu_paper} in \cite{10.1214/19-AAP1465}. 
\\$({\mathcal{A}}_{1})$ There exists constant $\kappa$, such that 
$|\bar{\sigma}(x,t)|^{2}\geq\kappa>0$, $\mathbb{P}$-a.s., for all $ (x,t) \in [0,T] \times \mathbb{R} $.
\\$({\mathcal{A}}_{2})$ The functions $\sigma$, $\bar{\sigma}$ and their spatial partial derivatives $\sigma_x$, $\bar{\sigma}_x$ are $\mathscr{P}\otimes\mathcal{B}(\mathbb{R})$-measurable and essentially bounded by a positive constant $K>0$. 
\\$({\mathcal{A}}_{3})$ For each $(u,v,\psi)\in (L^2([0,b]))^3$, we have $\Gamma(\cdot,u,v,\psi)\in {\cL^2_{\mathscr{F}}(0,T;L^{2}([0,b]))}$; and there exists a nonnegative constant $L$ such that for any $(u_i, v_i,\psi_{i})\in (L^2([0,b]))^3$, $i=1,2$ and any $t\in[0,T]$, there holds
\begin{equation*}
\| \Gamma(\cdot,t,u_{1},v_1,\psi_{1})
- \Gamma(\cdot,t,u_{2},v_2,\psi_{2}) \| \leq L\left(
\|u_{1}-u_{2}\| + \|v_{1}-v_{2}\| + \|\psi_{1}-\psi_{2}\|\right), 
\quad\mathbb{P}-\text{a.s.}
\end{equation*}
$({\mathcal{A}}_{4})$ The terminal $G$ is $\mathscr{F}_T\otimes\mathcal{B}(\mathbb{R})-$measurable and $G\in L^{2}\left(\Omega, \mathcal{F}_{T};
  H^{1,2}\left([0,b]\right)\right)$.

Then we introduce the notion of solutions to BSPDE \eqref{BSPDE_in_Qiu_paper}.
\begin{definition}\label{defsolution}
A pair of processes $(u,\psi)$ is a weak solution to BSPDE \eqref{BSPDE_in_Qiu_paper}, if $(u,\psi)\in\mathcal{H}$ with $u_x(t,\cdot)$ trace-zero at the boundary, and $(u,\psi)$ satisfies BSPDE \eqref{BSPDE_in_Qiu_paper} in the weak sense, that is, for any $\varphi\in C^\infty_0((0,b))$,
$$\langle\varphi\,,\, \frac{1}{2} \left( |\sigma|^2 
    + |\bar{\sigma}|^2 \right) u_{xx}
    + \sigma \psi_x
    + \Gamma(\cdot, u, u_x, \psi)\rangle\in\mathcal{L}^1_{\mathscr{F}}(0,T;\mathbb{R})$$ 
and for any $t\in[0,T]$,
\begin{equation*}
\langle\varphi,u(\cdot,t)\rangle=\langle\varphi,G\rangle+\int_t^T\langle\varphi \,,\, \frac{1}{2} \left( |\sigma|^2 
    + |\bar{\sigma}|^2 \right) u_{xx}
    + \sigma \psi_x
    + \Gamma(\cdot, u, u_x, \psi)\rangle (s)ds-\int_t^T\langle\varphi,\psi(\cdot,s)dW_s\rangle, \quad \mathbb{P}-a.s.
\end{equation*}
The above $(u,\psi)$ is called a strong solution if the regularity is improved $(u,\psi)\in\mathcal{H}^1$.
\end{definition}

Finally, we present the existence and uniqueness of strong solution to BSPDE \eqref{BSPDE_in_Qiu_paper} which can be found in \cite{10.1214/19-AAP1465} (see Theorem 3.1).
\begin{theorem}\label{strong_sol_of_stochastic_HJB_equation}
Let assumptions $({\mathcal{A}}_{1})-({\mathcal{A}}_{4})$ hold. 
Then BSPDE \eqref{BSPDE_in_Qiu_paper} with zero Neumann boundary condition admits a unique strong solution $(u,\psi)$ satisfying
    \begin{equation*}
    \| (u, \psi) \|_{\mathcal{H}^1} 
    \leq C\Big(\|G\|_
    {L^2(\Omega,\mathcal{F}_T;H^{1,2}([0,b]))} 
    + \|\Gamma^0\|_{\mathcal{L}_{\mathscr{F}}^2(0,T;L^2([0,b]))}\Big),
    \end{equation*}
where $\Gamma^0 := \Gamma(\cdot, 0,0,0)$, and the constant $C$ depends on $ L,\ \kappa,\ K$ and $T$. 
\end{theorem}

\begin{remark}\label{remark1}
If we add proper assumptions on the boundary condition $g$, the Neumann problem \eqref{BSPDE_in_Qiu_paper} can be equivalently reduced to the case with zero Neumann boundary condition (see pp. 2842 in~\cite{10.1214/19-AAP1465}). Moreover, the existence and uniqueness result for the strong solution to BSPDE \eqref{BSPDE_in_Qiu_paper} can established in more general case where $\Gamma$ can also depend on $u_{xx}$ and $\psi_x$, for the details, the readers are refered to Theorem 3.1 in~\cite{10.1214/19-AAP1465}.  
\end{remark}


\subsection{The LDG method}

For any given positive integer $N$, divide the space interval $[0,b]$ as 
\begin{equation*}
    0=x_{\frac{1}{2}}<x_{\frac{3}{2}}<\cdots<
    x_{N-\frac{1}{2}}<x_{N+\frac{1}{2}}=b.
\end{equation*}
We denote 
\begin{equation*}
    \begin{aligned}
    &I_{j} := 
    \big[x_{j-\frac{1}{2}},x_{j+\frac{1}{2}}\big],\quad
    x_{j}:=\frac{1}{2}\big(x_{j-\frac{1}{2}}+x_{j+\frac{1}{2}}\big),\quad h_{j} :=x_{j+\frac{1}{2}}-x_{j-\frac{1}{2}} ,\quad
    h:=\mathop{\text{max}}_{1\leq j\leq N}h_{j}. 
\end{aligned}
\end{equation*}
The finite element space is
\begin{equation*}
    V_{h}^k:= \Big\{
    v:v\mid_{I_{j}}\in P^{k}(I_{j}),\quad j=1,\cdots,N
    \Big\},
\end{equation*}
where $P^k(I_j)$ is the space of polynomials of total degree at most $k$. Note that functions in $V_{h}^k$ might have discontinuities on an element interface. We denote
$$
u_{j+\frac{1}{2}}^{-}=u\big(x_{j+\frac{1}{2}}^{-}\big),\quad
u_{j+\frac{1}{2}}^{+}=u\big(x_{j+\frac{1}{2}}^{+}\big),
$$
where $u$ is a given function defined on $[0,b]$.

\textbf{Gauss-Radau projection}
We consider the standard $L^2-$projection (denoted by $\mathcal{P}$) and the local Gauss-Radau projections (denoted by $\mathcal{P}^{\pm}$) into space $V^k_h$. For each $j=1,\cdots,N$, the projections satisfy
\begin{equation*}
    \int_{I_j} \left[\mathcal{P}u(x) - u(x)\right] v(x) dx = 0, 
    \quad \forall v \in P^k(I_j)
\end{equation*}
and
\begin{equation}\label{Gauss_Radou_projection_plus}
    \left\{
    \begin{aligned}
    &\int_{I_j} \left[\mathcal{P}^{+}u(x) - u(x)\right] v(x) dx = 0, 
    \quad \forall v \in P^{k-1}(I_j),\\
    & \mathcal{P}^{+}u(x_{j-\frac{1}{2}}^{+})=u(x_{j-\frac{1}{2}})
    \end{aligned}\right. 
\end{equation}
and
\begin{equation}\label{Gauss_Radou_projection_minus}
    \left\{
    \begin{aligned}
    &\int_{I_j} \left[\mathcal{P}^{-}u(x) - u(x)\right] v(x) dx = 0, 
    \quad \forall v \in P^{k-1}(I_j),\\
    & \mathcal{P}^{-}u(x_{j+\frac{1}{2}}^{-})=u(x_{j+\frac{1}{2}}). 
    \end{aligned}\right. 
\end{equation}
Then the projections have the following approximation error (see \cite{Ciarlet2002TheFE}),
\begin{equation}\label{Gauss_Radou_property}
    \|\mathcal{P}u-u\|+\|\mathcal{P}^{+}u-u\|
    +\|\mathcal{P}^{-}u-u\|
    \leq C \|u\|_{H^{k+1,2}} h^{k+1},
\end{equation}
where the constant $C$ does not depend on $h$ and $u$.

\section{The LDG method and the main results}
Due to Remark \ref{remark1}, to define the LDG method, we begin with considering the following equivalent first order system:
\begin{subnumcases}{}
    -d u(x,t)=\left[ p_{x}(x,t)+\lambda(x,t)v(x,t)+
    \mu(x,t)\psi(x,t)+\Gamma(x,t,u,v,\psi) \right]dt
    -\psi(x,t)dW_{t},\label{ldg_eq_a}\\
    v(x,t)=u_{x}(x,t),\label{ldg_eq_b}\\
    p(x,t)=\frac{1}{2}(\sigma^{2}(x,t)
    +\bar{\sigma}^{2}(x,t))v(x,t)
    +\sigma(x,t)\psi(x,t),\label{ldg_eq_c}\\
    v(t,0)=0,\quad v(t,b)=0,\\
    u(x,T)=G(x),\label{ldg_eq_terminal}
\end{subnumcases}
where for simplicity, we denote
\begin{align}
    \lambda(x,t)&:=-\sigma(x,t) \sigma_{x}(x,t)-
    \bar{\sigma}(x,t) \bar{\sigma}_{x}(x,t),\label{def_lambda}\\
    \mu(x,t)&:=-\sigma_{x}(x,t).\label{def_mu}
\end{align}

In order to define the approximating solution $(u_h,v_h,\psi_h,p_h)$, we first multiply \eqref{ldg_eq_a}--\eqref{ldg_eq_c} and \eqref{ldg_eq_terminal} by arbitrary smooth functions $z_{u},z_{v},z_{p},z_{G}$, respectively, and integrate over $I_{j}$, $j=1,\cdots,N$, then we obtain after a simple formal integration by parts that
\begin{align*}
        &-\int_{I_{j}}du(x,t)z_{u}(x) dx
        +\int_{I_{j}}\psi(x,t)z_{u}(x)dxdW_{t}\\\nonumber
        =&\left\{-\int_{I_{j}}p(x,t) (z_{u})_x(x)dx
        +p(x_{j+\frac{1}{2}},t)z_{u}(x_{j+\frac{1}{2}}^{-})
        -p(x_{j-\frac{1}{2}},t)z_{u}(x_{j-\frac{1}{2}}^{+}) + \int_{I_{j}} \left[ \lambda v + \mu\psi
        +\Gamma(\cdot,u,v,\psi)\right](x,t)\, z_{u}(x)dx \right\}dt,\\\nonumber
    &\int_{I_{j}}v(x,t) z_{v}(x)dx
        =-\int_{I_{j}}u(x,t) (z_{v})_x(x)dx
        +u(x_{j+\frac{1}{2}},t)z_{v}(x_{j+\frac{1}{2}}^{-})
        -u(x_{j-\frac{1}{2}},t)z_{v}(x_{j-\frac{1}{2}}^{+})\\\nonumber
        &\int_{I_{j}}p(x,t) z_{p}(x)dx=
        \int_{I_{j}}\left[\frac{1}{2}\left(\sigma^{2}(x,t)
        +\bar{\sigma}^{2}(x,t) \right)v(x,t)
        +\sigma(x,t)\psi(x,t)\right]z_{p}(x)dx,\\\nonumber
        &\int_{I_{j}}u(x,T)z_{G}(x)dx=\int_{I_{j}}G(x)z_{G}(x)dx.
\end{align*}
Then we replace the solutions $u,v,\psi,p$
by $u_{h},v_{h},\psi_{h},p_{h}$, 
and the test functions $z_{u},z_{v},z_{p},z_{G}$ by $z_{h,u},z_{h,v},z_{h,p},z_{h,G}$.

Since the functions in $V_{h}^k$ may be discontinuous at the endpoints, we need to replace the boundary terms obtained by integrating by parts with appropriate numerical fluxes. Precisely speaking, replace 
\begin{equation*}
    u(x_{j+\frac{1}{2}},t), 
    \quad\quad p(x_{j+\frac{1}{2}},t),\quad j=1,2,...,N,
\end{equation*}
with the following numerical fluxes
\begin{equation*}
    \hat{u}(x_{j+\frac{1}{2}},t), 
    \quad\quad \hat{p}(x_{j+\frac{1}{2}},t), \quad j=1,2,...,N,
\end{equation*}
with
\begin{equation}\label{flux_def_0}
    \begin{aligned}
    &\hat{u}(x_{j+\frac{1}{2}},t):=u_{h}(x_{j+\frac{1}{2}}^{-},t),
    \quad\quad \hat{p}(x_{j+\frac{1}{2}},t):=p_{h}(x_{j+\frac{1}{2}}^{+},t),
    \quad j=1,2...,N-1,\\
    &\hat{u}(x_{\frac{1}{2}},t)=\hat{u}(0,t):=u_{h}(0^{+},t),
    \quad\quad \hat{p}(x_{\frac{1}{2}},t)=\hat{p}(0,t):= p_{h}(0^{+},t) = 0,\\
    &\hat{u}(x_{N+\frac{1}{2}},t)=\hat{u}(b,t):=u_{h}(b^{-},t),
    \quad\quad \hat{p}(x_{N+\frac{1}{2}},t)=\hat{p}(b,t):= p_{h}(b^{-},t) = 0.
    \end{aligned}
\end{equation}
where the boundary of numerical flux is defined based on~\eqref{ldg_eq_c} and the Neumann boundary conditions and $\sigma(0,t)=\sigma(b,t)=0$, that is, the exact solution satisfies that $p(0,t)=p(b,t)=0$.
\begin{remark}
Alternatively, the numerical fluxes can be also defined as follows
\begin{equation*}
    \begin{aligned}
    &\hat{u}(x_{j+\frac{1}{2}},t):=u_{h}(x_{j+\frac{1}{2}}^{+},t),
    \quad\quad \hat{p}(x_{j+\frac{1}{2}},t):=p_{h}(x_{j+\frac{1}{2}}^{-},t),
    \quad j=1,2...,N-1.
    \end{aligned}
\end{equation*}
\end{remark}
Then, the approximating solution given by the LDG method is defined as the solution of the following weak formulation: for any $z_{h,u},\ z_{h,v},\ z_{h,p}$ in $V_h^k$, it holds that
\begin{subnumcases}{}\label{approx_eq}
        \int_{I_{j}}du_{h}(x,t)z_{h,u}(x)dx=-\Bigg\{ -\int_{I_{j}}p_{h}(x,t)(z_{h,u})_x(x)dx
        +\hat{p}(x_{j+\frac{1}{2}},t)z_{h,u}(x_{j+\frac{1}{2}}^{-})
        -\hat{p}(x_{j-\frac{1}{2}},t)z_{h,u}(x_{j-\frac{1}{2}}^{+})\nonumber\\
        \quad\quad\quad\quad\quad\quad\quad\quad\quad\quad\quad +\int_{I_{j}}\left[\lambda v_{h} + \mu\psi_{h} + \Gamma(\cdot,u_{h},v_{h},\psi_{h}) \right](x,t) z_{h,u}(x)dx \Bigg\}dt +\int_{I_{j}}\psi_{h}(x,t)z_{h,u}(x)dxdW_{t},\label{approx_eq_a}\\
        \int_{I_{j}}v_{h}(x,t)z_{h,v}(x)dx=-\int_{I_{j}}u_{h}(x,t)(z_{h,v})_x(x)dx +\hat{u}(x_{j+\frac{1}{2}},t)z_{h,v}(x_{j+\frac{1}{2}}^{-})-\hat{u}(x_{j-\frac{1}{2}},t)z_{h,v}(x_{j-\frac{1}{2}}^{+}),\label{approx_eq_b}\\
        \int_{I_{j}}p_{h}(x,t)z_{h,p}(x)dx=\int_{I_{j}}\left[\frac{1}{2}\left(\sigma^{2}(x,t)+\bar{\sigma}^{2}(x,t)\right)v_{h}(x,t) +\sigma(x,t)\psi_{h}(x,t)\right]z_{h,p}(x)dx,\label{approx_eq_c}\\
        \int_{I_{j}}u_{h}(x,T)z_{h,G}(x)dx=\int_{I_{j}}G(x)z_{h,G}(x)dx.\label{approx_eq_d}
\end{subnumcases}

The following theorem concerns on the well-posedness of the approximating equation~\eqref{approx_eq_a}-\eqref{approx_eq_d}.
\begin{theorem}\label{uniqueness_and_existence_of_approx_equation} 
Let assumptions $(\mathcal{A}_{1})-(\mathcal{A}_{4})$ hold. Then for each fixed $N\in\mathbb{N}_{+}$, 
    the approximate equation \eqref{approx_eq} admits a unique solution 
    $(u_{h},v_{h},p_{h},\psi_{h}) \in (V_{h})^{4}$.
\end{theorem}
Now we state the stability result for the numerical solutions.
\begin{theorem}\label{theorem_ldg_stability}
    Let assumptions $(\mathcal{A}_{1})-(\mathcal{A}_{4})$ hold. 
    Then there exists a constant $C$ that does not depend on $h$ such that
    \begin{equation*}
        \begin{aligned}
        &\quad\mathbb{E}\Big[
        \sup_{0 \leq t \leq T}
        \|u_{h}(\cdot,t)\|^{2}\Big]
        +\mathbb{E}\Big[
        \int_{0}^{T}\|v_{h}(\cdot,s)\|^{2}ds\Big]
        +\mathbb{E}\Big[
        \int_{0}^{T}\|\psi_{h}(\cdot,s)\|^{2}ds
        \Big]\leq
        C\left(\mathbb{E}\Big[
        \|G(\cdot)\|^{2}\Big]
        +\mathbb{E}\Big[
        \int_{0}^{T}\|\Gamma^{0}(\cdot,s)\|^{2}ds
        \Big]\right).
        \end{aligned}
    \end{equation*}
\end{theorem}
Finally, we give the optimal error estimate of our numerical method under the following assumptions. 
\\$(\mathcal{A}_{5})$\ The coefficients $\sigma$ and $\bar{\sigma}$ and their up-to $k+2-$th order derivatives are essentially bounded, 
and for any $(u,\psi)\in H^{k+3,2}\times H^{k+2,2}$,  $\Gamma(\cdot,u,u_x,\psi)\in\mathcal{L}^2_{\mathscr{F}}(0,T;H^{k+1,2})$ and satisfies Lipschitz condition, i.e. for any $(u^i,\psi^i)\in H^{k+3,2}\times H^{k+2,2}$, $i=1,2$, there exists a nonnegative constant $L$ such that 
$$\|\Gamma_t(\cdot,u^1,u^1_x,\psi^1)-\Gamma_t(\cdot,u^2,u^2_x,\psi^2)\|_{H^{k+1,2}}\leq L(\|u^1-u^2\|_{H^{k+2,2}}+\|\psi^1-\psi^2\|_{H^{k+1,2}}),\quad a.s.,$$
for any $t\in[0,T]$;
\\$(\mathcal{A}_{6})$\ The terminal value $G$ is in $L^2(\Omega,\mathscr{F}_T;H^{k+3,2})$;
\\$(\mathcal{A}_{7})$\ The boundary value $g_T$ is in $L^2(\Omega,\mathscr{F}_T;H^{k+2,2})$ and together with another function $\mathcal{G}$, such that $(g,\mathcal{G}) \in \mathcal{H}^{k+2}$ and satisfies BSPDE $-dg_t=\mathscr{G}_tdt-\mathcal{G}_tdW_t$ in the weak sense (see Definition \ref{defsolution}) with $\mathscr{G}\in\mathcal{L}^2_{\mathscr{F}}(0,T;H^{k+1,2})$.

\begin{lemma}\label{highordersolution}
Under assumption $(\mathcal{A}_{5})-(\mathcal{A}_{7})$, BSPDE \eqref{BSPDE_in_Qiu_paper} has a unique solution $(u,\psi)\in\mathcal{H}^{k+2}$.
\end{lemma}

Under $(\mathcal{A}_{6})$ and $(\mathcal{A}_{7})$, BSPDE \eqref{BSPDE_in_Qiu_paper} can still be transformed to the zero Neumann boundary case, then by the continuity method, we can prove Lemma \ref{highordersolution}, so we omit it here. 

\begin{theorem}\label{theorem_ldg_optimal_error}
Let assumptions $(\mathcal{A}_{5})-(\mathcal{A}_{7})$ be satisfied and $\sigma(0,t)=\sigma(b,t)=0$. There exists a constant $C$ independent with $h$ such that
\begin{equation*}
\begin{aligned}
  &\quad\mathbb{E}\Big[\sup_{0\leq t\leq T}
  \|u(\cdot,t)-u_{h}(\cdot,t)\|^{2}\Big]^{\frac{1}{2}}
  +\mathbb{E}\Big[\int_{0}^{T}
  \|u_{x}(\cdot,s)-v_{h}(\cdot,s)\|^{2}ds\Big]^{\frac{1}{2}}
  +\mathbb{E}\Big[\int_{0}^{T}
  \|\psi(\cdot,s)-\psi_{h}(\cdot,s)\|^{2}ds\Big]^{\frac{1}{2}}
  \leq Ch^{k+1}.
\end{aligned}
\end{equation*}
\end{theorem}

\section{Proofs}
\subsection{The well-posedness of the approximating equations}
In order to prove the existence and uniqueness of the solution of the approximation equation, we first expand the LDG numerical solution on each finite element $I_{j}$. For $\varphi:=u,v,\psi,p$ and $\boldsymbol{\varphi}:=\boldsymbol{u}, \boldsymbol{v}, \boldsymbol{\psi}, \boldsymbol{p}$, we have
\begin{equation}\label{approx_sol_expression}
\varphi_{h}(\omega, x,t) =\sum_{l=0}^{k}\boldsymbol{\varphi}_{h}^{l,j}
        (\omega,t)\phi_{l}^{j}(x),\quad\quad x\in I_{j},\quad j=1, 2,...,N,
\end{equation}
where $\{\phi_{l}^{j},\,j=0,1,...,k\}$ is a set of polynomial bases on the interval $I_{j}$.
Based on the definition of numerical flux~\eqref{flux_def_0}, we define
\begin{equation*}
    \begin{aligned}
    &\boldsymbol{u}_{h}^{l,0}(t)
    :=\boldsymbol{u}_{h}^{l,1}(t),\quad\quad
    \boldsymbol{p}_{h}^{l,0}(t)
    :=\boldsymbol{p}_{h}^{l,1}(t),\quad\quad
    \boldsymbol{u}_{h}^{l,N+1}(t)
    :=\boldsymbol{u}_{h}^{l,N}(t),\quad\quad
    \boldsymbol{p}_{h}^{l,N+1}(t)
    :=\boldsymbol{p}_{h}^{l,N}(t).
    \end{aligned}
\end{equation*}
Then for $\boldsymbol{\varphi}=\boldsymbol{u}, \boldsymbol{v}, \boldsymbol{\psi}, \boldsymbol{p}$, we denote 
\begin{equation*}
    \begin{aligned}
        \boldsymbol{\varphi}_t:=\left\{
        \boldsymbol{\varphi}_{h}^{l,j}(\omega,t)\right\}
        _{l\in \left\{ 0,1,...k\right\},j\in \left\{ 1,2,...N\right\}}.
    \end{aligned}
\end{equation*}

For $j=1,2,...N$, take $z_{h,v}=\phi_{m}^{j}$ in~\eqref{approx_eq_b} to get
\begin{equation*}
    \begin{aligned}
        &\quad \sum_{l=0}^{k}\Big(
        \int_{I_{j}}\phi_{m}^{j}(x)\phi_{l}^{j}(x)dx
        \Big)\boldsymbol{v}_{h}^{l,j}(t)\\
        &= -\int_{I_{j}}\sum_{n=0}^{k}\boldsymbol{u}_{h}^{n,j}(t)
        \phi_{n}^{j}(x)(\phi_{m}^{j})_x (x)dx + \sum_{n=0}^{k}\boldsymbol{u}_{h}^{n,j}(t)
        \phi_{n}^{j}(x_{j+\frac{1}{2}})\phi_{m}^{j}(x_{j+\frac{1}{2}})
        -\sum_{n=0}^{k}\boldsymbol{u}_{h}^{n,j-1}(t)
        \phi_{n}^{j-1}(x_{j-\frac{1}{2}})\phi_{m}^{j}(x_{j-\frac{1}{2}}).
\end{aligned}
\end{equation*}
Define mass matrix $A^{j}:=\left(A^j_{m,l}\right)$ with
\begin{equation*}
    A^j_{m,l}=\int_{I_{j}}\phi_{m}^{j}(x)\phi_{l}^{j}(x)dx.
\end{equation*}
From the properties of the polynomial basis, it is easy to know that the mass matrix $A^{j}$ is invertible, its inverse matrix is denoted as $A^{j,-1}$, and the $(m,l)$ element of its inverse matrix is denoted as $A^{j,-1}_{m,l}$. Then we obtain
\begin{equation}\label{approx_v_of_u}
    \boldsymbol{v}_{h}^{l,j}(t)=\boldsymbol{V}_{h}^{l,j}\left( \boldsymbol{u}_t \right),
    \quad l=0,1,...,k,\quad j=1,2,...,N,
\end{equation}
with
\begin{equation*}
    \begin{aligned}
    \boldsymbol{V}_{h}^{l,j}(\boldsymbol{u})=
    &-\int_{I_{j}}\sum_{n=0}^{k}\boldsymbol{u}_{h}^{n,j}
    \phi_{n}^{j}(x)
    \sum_{m=0}^{k}A_{l,m}^{j,-1} (\phi_{m}^{j})_x(x)dx\\
    &+\sum_{n=0}^{k}\boldsymbol{u}_{h}^{n,j}
    \phi_{n}^{j}(x_{j+\frac{1}{2}})
    \sum_{m=0}^{k}A_{l,m}^{j,-1}\phi_{m}^{j}(x_{j+\frac{1}{2}})-\sum_{n=0}^{k}\boldsymbol{u}_{h}^{n,j-1}
    \phi_{n}^{j-1}(x_{j-\frac{1}{2}})
    \sum_{m=0}^{k}A_{l,m}^{j,-1}\phi_{m}^{j}(x_{j-\frac{1}{2}}).  
    \end{aligned}
\end{equation*}

For $j=1,2,...N$, take $z_{h,p}=\phi_{m}^{j}$ in~\eqref{approx_eq_c} to get
\begin{equation*}
    \begin{aligned}
        \sum_{l=0}^{k}\Big(
        \int_{I_{j}}\phi_{m}^{j}(x)\phi_{l}^{j}(x)dx
        \Big)\boldsymbol{p}_{h}^{l,j}(t)
        =\int_{I_{j}}
        \left[\frac{1}{2}\left(\sigma^{2}(x,t)
        +\bar{\sigma}^{2}(x,t)\right)
        \sum_{n=0}^{k}\boldsymbol{v}_{h}^{n,j}(t)
        \phi_{n}^{j}(x)
        +\sigma(x,t)\sum_{n=0}^{k}
        \boldsymbol{\psi}_{h}^{n,j}(t)
        \phi_{n}^{j}(x)\right]\phi_{m}^{j}(x)dx.
    \end{aligned}
\end{equation*}
Then we have 
\begin{equation}\label{approx_p_of_v_psi}
    \boldsymbol{p}_{h}^{l,j}(t)=\boldsymbol{P}_{h}^{l,j}
    \left(t,\boldsymbol{u}_{t},\boldsymbol{\psi}_{t}\right),
    \quad l=0,1,...,k,\quad j=1,2,...,N,
\end{equation}
with
\begin{equation*}
    \begin{aligned}
    \boldsymbol{P}_{h}^{l,j}(t,\boldsymbol{u},\boldsymbol{\psi})
        =&\int_{I_{j}}
        \left[\frac{1}{2}\left(\sigma^{2}(x,t)
        +\bar{\sigma}^{2}(x,t)\right)
        \sum_{n=0}^{k}\boldsymbol{v}_{h}^{n,j}(t)
        \phi_{n}^{j}(x)
        +\sigma(x,t)\sum_{n=0}^{k}
        \boldsymbol{\psi}_{h}^{n,j}(t)
        \phi_{n}^{j}(x)\right]\sum_{m=0}^{k}A_{l,m}^{j,-1}
        \phi_{m}^{j}(x)dx. 
    \end{aligned}
\end{equation*}

For $j=1,2,...N$, take $z_{h,u}=\phi_{m}^{j}$ in~\eqref{approx_eq_a} to get
    \begin{align*}
        &\quad \sum_{l=0}^{k}\Big(
        -\int_{I_{j}}\phi_{m}^{j}(x)\phi_{l}^{j}(x)dx
        \Big)d\boldsymbol{u}_{h}^{l,j}(t)\\
        =&\Bigg\{
        -\int_{I_{j}}
        \sum_{n=0}^{k}\boldsymbol{p}_{h}^{n,j}(t)
        \phi_{n}^{j}(x)(\phi_{m}^{j})_x(x)dx + \sum_{n=0}^{k}\boldsymbol{p}_{h}^{n,j+1}(t)
        \phi_{n}^{j+1}(x_{j+\frac{1}{2}})
        \phi_{m}^{j}(x_{j+\frac{1}{2}})
        -\sum_{n=0}^{k}\boldsymbol{p}_{h}^{n,j}(t)
        \phi_{n}^{j}(x_{j-\frac{1}{2}})
        \phi_{m}^{j}(x_{j-\frac{1}{2}})\\
        &\quad+\int_{I_{j}}\sum_{n=0}^{k}\boldsymbol{v}_{h}^{n,j}(t)
        \phi_{n}^{j}(x)\phi_{m}^{j}(x)\lambda(x,t)dx
        +\int_{I_{j}}\sum_{n=0}^{k}\boldsymbol{\psi}_{h}^{n,j}(t)
        \phi_{n}^{j}(x)\phi_{m}^{j}(x)\mu(x,t)dx\\
        &\quad+\int_{I_{j}}\Gamma \big(x,t,
        \sum_{n=0}^{k}\boldsymbol{u}_{h}^{n,j}(t)\phi_{n}^{j}(x),
        \sum_{n=0}^{k}\boldsymbol{v}_{h}^{n,j}(t)\phi_{n}^{j}(x),
        \sum_{n=0}^{k}\boldsymbol{\psi}_{h}^{n,j}(t)\phi_{n}^{j}(x)\big)
        \phi_{m}^{j}(x)dx\Bigg\}dt -\int_{I_{j}}\sum_{n=0}^{k}\boldsymbol{\psi}_{h}^{n,j}(t)
        \phi_{n}^{j}(x)\phi_{m}^{j}(x)dxdW_{t}.
    \end{align*}
Hence, we obtain the following BSDE satisfied by $\left(\boldsymbol{u}_{t},\boldsymbol{\psi}_{t}\right)$:
\begin{equation}\label{approx_bsde}
    -d\boldsymbol{u}_{t}=\boldsymbol{F_{h}}(t,\boldsymbol{u}_{t},\boldsymbol{\psi}_{t})dt
    -\boldsymbol{\psi}_{t}dW_{t},
\end{equation}
where for $l=0,1,...,k,\ j=1,2,...,N$, the generator is defined as follows
\begin{align*}
        \boldsymbol{F}_{h}^{l,j}(t,\boldsymbol{u},\boldsymbol{\psi})
        =&-\int_{I_{j}}
        \sum_{n=0}^{k}\boldsymbol{P}_{h}^{n,j}
        (t,\boldsymbol{u},\boldsymbol{\psi})
        \phi_{n}^{j}(x)\sum_{m=0}^{k}A_{l,m}^{j,-1}
        (\phi_{m}^{j})_x (x)dx
        +\sum_{n=0}^{k}\boldsymbol{P}_{h}^{n,j+1}
        (t,\boldsymbol{u},\boldsymbol{\psi})
        \phi_{n}^{j+1}(x_{j+\frac{1}{2}})
        \sum_{m=0}^{k}A_{l,m}^{j,-1}
        \phi_{m}^{j}(x_{j+\frac{1}{2}})\\
        &-\sum_{n=0}^{k}\boldsymbol{P}_{h}^{n,j}
        (t,\boldsymbol{u},\boldsymbol{\psi})
        \phi_{n}^{j}(x_{j-\frac{1}{2}})
        \sum_{m=0}^{k}A_{l,m}^{j,-1}
        \phi_{m}^{j}(x_{j-\frac{1}{2}})
        +\int_{I_{j}}\sum_{n=0}^{k}\boldsymbol{V}_{h}^{n,j}(\boldsymbol{u})
        \phi_{n}^{j}(x)\sum_{m=0}^{k}A_{l,m}^{j,-1}\phi_{m}^{j}(x)\lambda(x,t)dx\\
        &+\int_{I_{j}}\sum_{n=0}^{k}\boldsymbol{\psi}_{h}^{n,j}
        \phi_{n}^{j}(x)\sum_{m=0}^{k}A_{l,m}^{j,-1}\phi_{m}^{j}(x)\,\mu(x,t)dx\\
        &+\int_{I_{j}}\Gamma\big(x,t,
        \sum_{n=0}^{k}\boldsymbol{u}_{h}^{n,j}\phi_{n}^{j}(x),
        \sum_{n=0}^{k}\boldsymbol{V}_{h}^{n,j}(\boldsymbol{u})\phi_{n}^{j}(x),
        \sum_{n=0}^{k}\boldsymbol{\psi}_{h}^{n,j}\phi_{n}^{j}(x)\big)
        \sum_{m=0}^{k}A_{l,m}^{j,-1}\phi_{m}^{j}(x)dx.
\end{align*}

\begin{lemma}\label{proof_of_lipchitz}
    Suppose $(\mathcal{A}_{1})-(\mathcal{A}_{4})$ hold, then for any given $N$, 
    the coefficient $\boldsymbol{F}_{h}$ is uniformly Lipschitz continuous with respect to $\boldsymbol{u}$ and $\boldsymbol{\psi}$. That is, there exists a positive real number $C_{N}$, such that for any 
    $(\omega,t)\in\Omega\times[0,T]$ and any $\boldsymbol{u},\boldsymbol{\tilde{u}},
\boldsymbol{\psi},\boldsymbol{\tilde{\psi}}
\in \mathbb{R}^{(k+1)\times N}$, it holds
    \begin{equation*}
        |\boldsymbol{F_{h}}(t,\boldsymbol{u},\boldsymbol{\psi})
        -\boldsymbol{F_{h}}(t,\boldsymbol{\tilde{u}},
        \boldsymbol{\tilde{\psi}})|
         \leq C_{N}
        \left(|\boldsymbol{u}-\boldsymbol{\tilde{u}}|+
        |\boldsymbol{\psi}-\boldsymbol{\tilde{\psi}}|\right).
    \end{equation*}
\end{lemma}
\begin{proof}
For given $N$, we have that the function $\boldsymbol{V}$ is uniformly Lipschitz continuous, i.e., for any $\boldsymbol{u},\boldsymbol{\tilde{u}}\in 
\mathbb{R}^{(k+1)\times N}$,
\begin{equation}\label{lipchitz_for_v}
    \begin{aligned}
        |\boldsymbol{V}(\boldsymbol{u})
        -\boldsymbol{V}(\boldsymbol{\tilde{u}})|^{2}
        =\sum_{l=0}^{k}\sum_{j=1}^{N}
        |\boldsymbol{{V}}_h^{l,j}(\boldsymbol{{u}})
        -\boldsymbol{V}_{h}^{l,j}(\boldsymbol{\tilde{u}})|^{2}
        &\leq \sum_{l=0}^{k}\sum_{j=1}^{N} C_{N}
        \sum_{n=0}^{k}\left(
        |\boldsymbol{u}_{h}^{n,j}-{\boldsymbol{\tilde{u}}_{h}^{n,j}}|^{2}
        +|\boldsymbol{u}_{h}^{n,j-1}-\boldsymbol{\tilde{u}}_{h}^{n,j-1}|^{2}
        \right) \leq C_{N} \big|\boldsymbol{u}-\boldsymbol{\tilde{u}}\big|^2.
    \end{aligned}
\end{equation}
Now we have that $\boldsymbol{P}$ is Lipschitz continuous
\begin{equation}\label{lipchitz_for_p}
    \begin{aligned}
    |\boldsymbol{P}(t,\boldsymbol{u},\boldsymbol{\psi})
    -\boldsymbol{P}(t,\boldsymbol{\tilde{u}}
    ,\boldsymbol{\tilde{\psi}})|^{2} \leq \sum_{j=1}^{N}\sum_{l=0}^{k} C_{N} \sum_{n=0}^{k}
    \left(
    |\boldsymbol{V}_{h}^{n,j}(\boldsymbol{u})
    -\boldsymbol{V}_{h}^{n,j}(\boldsymbol{\tilde{u}})|^{2}
    +|\boldsymbol{\psi}_{h}^{n,j}
    -\boldsymbol{\tilde{\psi}}_{h}^{n,j}|^{2}\right)
    \leq C_{N}\left(
        |\boldsymbol{u}-\boldsymbol{\tilde{u}}|^{2}
        +|\boldsymbol{\psi}-\boldsymbol{\tilde{\psi}}|^{2}\right).
    \end{aligned}
\end{equation}
Finally, we prove that the coefficients of BSDE \eqref{approx_bsde} are Lipschitz continuous with respect to $\boldsymbol{u},\boldsymbol{\psi}$. For any 
$ \boldsymbol{u},\boldsymbol{\tilde{u}},\boldsymbol{\psi},\boldsymbol{\tilde{\psi}}\in 
\mathbb{R}^{(k+1)\times N}$, $l=0,1,...,k$, $j=1,2,...N$, due to $(\mathcal{A}_{3})$ and Cauchy-Schwartz's inequality, it holds that
\begin{align}
    |\boldsymbol{F}_{h}^{l,j}(t,\boldsymbol{u},\boldsymbol{\psi})
    -\boldsymbol{F}_{h}^{l,j}(t,\boldsymbol{\tilde{u}},\boldsymbol{\tilde{\psi}})|^{2}&\nonumber\leq C_{N}\Big(\sum_{n=0}^{k}
    |\boldsymbol{P}_{h}^{n,j}(t,\boldsymbol{u},\boldsymbol{\psi})
    -\boldsymbol{P}_{h}^{n,j}(
    t,\boldsymbol{\tilde{u}},\boldsymbol{\tilde{\psi}})|^{2}
    +\sum_{n=0}^{k}
    |\boldsymbol{P}_{h}^{n,j+1}(t,\boldsymbol{u},\boldsymbol{\psi})
    -\boldsymbol{P}_{h}^{n,j+1}(
    t,\boldsymbol{\tilde{u}},\boldsymbol{\tilde{\psi}})|^{2}\\
    &\nonumber\quad\quad\quad +\sum_{n=0}^{k}
    |\boldsymbol{V}_{h}^{n,j}(\boldsymbol{u})
    -\boldsymbol{V}_{h}^{n,j}(
    \boldsymbol{\tilde{u}})|^{2}
    +\sum_{n=0}^{k}
    |\boldsymbol{\psi}_{h}^{n,j}
    -\boldsymbol{\tilde{\psi}}_{h}^{n,j}|^{2}
    +\sum_{n=0}^{k}
    |\boldsymbol{u}_{h}^{n,j}
    -{\boldsymbol{\tilde{u}}_{h}^{n,j}}|^{2}\Big)\\
    &\nonumber\leq C_{N}
    \Big(\sum_{n=0}^{k}
    |\boldsymbol{u}_{n,j+1}
    -\boldsymbol{\tilde{u}}_{n,j+1}|^{2}
    +\sum_{n=0}^{k}
    |\boldsymbol{u}_{h}^{n,j}
    -{\boldsymbol{\tilde{u}}_{h}^{n,j}}|^{2}
    +\sum_{n=0}^{k}
    |\boldsymbol{u}_{h}^{n,j-1}
    -\boldsymbol{\tilde{u}}_{h}^{n,j-1}|^{2}\\
    &\quad\quad\quad +\sum_{n=0}^{k}
    |\boldsymbol{\psi}_{n,j+1}
    -\boldsymbol{\tilde{\psi}}_{n,j+1}|^{2}
    +\sum_{n=0}^{k}
    |\boldsymbol{\psi}_{h}^{n,j}
    -\boldsymbol{\tilde{\psi}}_{h}^{n,j}|^{2}\Big),
    \label{ldg_bsde_coff_inequal_a}
\end{align}
which gives that
\begin{equation}
|\boldsymbol{F}_{h}(t,\boldsymbol{u},\boldsymbol{\psi})-
\boldsymbol{F}_{h}(t,\boldsymbol{\tilde{u}},\boldsymbol{\tilde{\psi}})|
\leq C_{N} \left(
    |\boldsymbol{u}-\boldsymbol{\tilde{u}}|
    +|\boldsymbol{\psi}-\boldsymbol{\tilde{\psi}}|\right).
\end{equation}
\end{proof}

\noindent\textbf{Proof of Theorem \ref{uniqueness_and_existence_of_approx_equation}}
By \eqref{approx_eq_d} the terminal condition of BSDE\eqref{approx_bsde} is as following
\begin{equation}
    \boldsymbol{u}_{h}^{l,j}(T):=\sum_{m=0}^{k}A_{l,m}^{j,-1}
    \int_{I_{j}}G(x)\phi_{m}^{j}(x)dx.
\end{equation}
From $(\mathcal{A}_{4})$ and Cauchy-Schwartz's inequality, we know that 
$\boldsymbol{u}_{T}\in L^{2}(\Omega,\mathcal{F}_{T};\mathbb{R}^{(k+1)\times N})$.
Combining Lemma \ref{proof_of_lipchitz} and the theory of BSDEs (see e.g. \cite[Theorem 4.3.1]{zhang2017backward}) yields that  
for each $N\in \mathbb{N}_{+}$, BSDE~\eqref{approx_bsde} admits a unique solution
$(\boldsymbol{u}_{t},\boldsymbol{\psi}_{t})$ satisfying 
\begin{equation}\label{bound_for_element_u_psi}
    \mathbb{E}\Big[
    \mathop{\text{sup}}_{0\leq t \leq T}|\boldsymbol{u}_{t}|^{2}\Big]
    +\mathbb{E}\Big[
    \int_{0}^{T}|\boldsymbol{\psi}_{t}|^{2}dt
    \Big]
    < \infty.
\end{equation}
Plugging it into \eqref{approx_sol_expression} provides that \eqref{approx_eq} has a unique solution $(u_{h},v_{h},p_{h},\psi_{h}) \in (V_{h})^{4}$. $\hspace{3.9cm}\Box$

\begin{remark}
It can be known from the linear growth of functions $\boldsymbol{V}$ and $\boldsymbol{P}$ with respect to $\boldsymbol{u}$ and $\boldsymbol{\psi}$,
\begin{equation}
    \mathbb{E}\Big[
    \mathop{\text{sup}}_{0\leq t \leq T}|\boldsymbol{v}_{t}|^{2}\Big]
    < \infty,
    \qquad
    \mathbb{E}\Big[
    \int_{0}^{T}|\boldsymbol{p}_{t}|^{2}dt
    \Big]
    < \infty.
\end{equation}
The above result plays an important role in the proofs of stability and convergence errors.
\end{remark}

\subsection{The stability of the LDG method}

Now we give the proof of the stability of the semi-discontinuous LDG method.

\begin{proof}[Proof of Theorem \ref{theorem_ldg_stability}] The proof is divided into two steps. 

Step 1. We are going to prove the following estimate:
\begin{equation*}
    \begin{aligned}
    &\quad\mathop{\text{sup}}_{0\leq t \leq T}
    \mathbb{E}\Big[\|u_{h}(\cdot,t)^{2}|\Big]
    +\mathbb{E}\Big[
    \int_{0}^{T}\|v_{h}(\cdot,s)\|^{2}ds\Big]
    +\mathbb{E}\Big[
    \int_{0}^{T}\|\psi_{h}(\cdot,s)\|^{2}ds
    \Big] \leq C\left(\mathbb{E}\Big[
    \|G(\cdot)\|^{2}\Big]
    +\mathbb{E}\Big[
    \int_{0}^{T} \|\Gamma^{0}(\cdot,s)\|^2 ds
    \Big]\right).
    \end{aligned}
\end{equation*}
Apply It\^o's formula to $u_{h}(x,t)$, we have 
\begin{equation}\label{stability_ito}
    |u_{h}(x,T)|^{2}-|u_{h}(x,t)|^{2}
    =2\int_{t}^{T}u_{h}(x,s)du_{h}(x,s)
    +\left<u_{h}(x,\cdot),u_{h}(x,\cdot)\right>_{t}^{T}.
\end{equation}
Taking the test functions $z_{h,u},z_{h,v}$ in the approximation equations~\eqref{approx_eq_a} and \eqref{approx_eq_b} as $u_{h},p_{h}$ respectively, and integrating both sides with respect to time from $t$ to $T$, and summing the indicator $j$, and replacing the terminal condition $u_{h}(x,T)$ in \eqref{stability_ito} by $G(x)$, and integrating with respect to the space $[0,b]$, then taking the expectations on both sides, we get
    \begin{align}
    &\|u_{h}(\cdot,r)\|^{2}+
    \int_{r}^{T}\int_{0}^{b}\left[\left(\sigma^{2}(x,s)
    +\bar{\sigma}^{2}(x,s) \right)v_{h}(x,s)
    +2\sigma(x,s) \psi_{h}(x,s)\right]v_{h}(x,s)dxds \nonumber\\
    =&\|G(\cdot)\|^{2}-\int_{0}^{b}
    \left<u_{h}(x,\cdot),u_{h}(x,\cdot)\right>_{r}^{T}dx
    -2\int_{r}^{T}\int_{0}^{b}
    \psi_{h}(x,s)u_{h}(x,s)dxdW_{s}\nonumber\\
    &+2\int_{r}^{T}\int_{0}^{b} \left[\lambda v_{h}u_{h} + \mu\psi_{h}u_{h} + \Gamma(\cdot,u_{h},v_{h},\psi_{h}) u_{h} \right](x,s) dxds\nonumber\\
    &+2\int_{t}^{T}\sum_{j=1}^{N}\Bigg( -\int_{I_{j}}p_{h}(x,s)(u_h)_x(x,s)dx
        +\hat{p}(x_{j+\frac{1}{2}},s)u_{h}(x_{j+\frac{1}{2}}^{-},s)
        -\hat{p}(x_{j-\frac{1}{2}},s)u_{h}(x_{j-\frac{1}{2}}^{+},s) \nonumber\\
    &\quad\quad\quad\quad\quad\quad -\int_{I_{j}}u_{h}(x,s)(p_{h})_x(x,s)dx +\hat{u}(x_{j+\frac{1}{2}},s)p_{h}(x_{j+\frac{1}{2}}^{-},s)-\hat{u}(x_{j-\frac{1}{2}},s)p_{h}(x_{j-\frac{1}{2}}^{+},s) \Bigg)ds.\label{stability_equation}
\end{align}
It follows that
\begin{equation*}
    \begin{aligned}
    \mathbb{E}\Big[\|u_{h}(\cdot,t)\|^{2}\Big]
    +2\mathbb{E}\Big[
    \int_{t}^{T}\int_{0}^{b}v_{h}(x,s)p_{h}(x,s)dxds\Big]
    =\mathbb{E}\Big[\|G\|^{2}\Big]
+\mathcal{T}_{1}(t)+\mathcal{T}_{2}(t)
    +\mathcal{T}_{3}(t)+\mathcal{T}_{4}(t),
    \end{aligned}
\end{equation*}
with
    \begin{align*}
    &\mathcal{T}_{1}(t):=-\mathbb{E}\Big[\int_{0}^{b}
    \left<u_{h}(x,\cdot),u_{h}(x,\cdot)\right>_{t}^{T} dx\Big],\\
    &\mathcal{T}_{2}(t):=2\mathbb{E}\Big[
    \int_{t}^{T}\int_{0}^{b}\left[\lambda v_{h}u_{h} + \mu\psi_{h}u_{h} + \Gamma(\cdot,u_{h},v_{h},\psi_{h}) u_{h} \right](x,s) dxds\Big],\\
    &\mathcal{T}_{3}(t)
    :=-2\mathbb{E}\Big[
    \int_{t}^{T}\int_{0}^{b}
    \psi_{h}(x,s)u_{h}(x,s)dxdW_{s}\Big],\\
    &\mathcal{T}_{4}(t)
    :=2\mathbb{E}\Bigg[
    \int_{t}^{T}\sum_{j=1}^{N}\bigg\{
    -\int_{I_{j}}p_{h}(x,s)(u_h)_x(x,s)dx
        +\hat{p}(x_{j+\frac{1}{2}},s)u_{h}(x_{j+\frac{1}{2}}^{-},s)
        -\hat{p}(x_{j-\frac{1}{2}},s)u_{h}(x_{j-\frac{1}{2}}^{+},s) \\
    & \quad\quad\quad\quad\quad\quad\quad\quad\quad -\int_{I_{j}}u_{h}(x,s)(p_{h})_x(x,s)dx +\hat{u}(x_{j+\frac{1}{2}},s)p_{h}(x_{j+\frac{1}{2}}^{-},s)-\hat{u}(x_{j-\frac{1}{2}},s)p_{h}(x_{j-\frac{1}{2}}^{+},s) \bigg\}ds \Bigg].
    \end{align*}
Then we estimate $\mathcal{T}_{j}(t),\ j=1,2,3,4$ in turn.

\textbf{Estimation of $\mathcal{T}_{1}(t)$.}

It can be known from the approximation equation \eqref{approx_eq_a} satisfied by $u_{h}(x,t)$ and equation \eqref{approx_eq_d} satisfied by the terminal term that
\begin{equation*}
    \begin{aligned}
    \int_{I_{j}}u_h(x,t)z_{h,u}(x)dx
    =&\int_{I_{j}}G(x)z_{h,u}(x)dx + \int_{t}^{T}\Bigg\{-\int_{I_{j}}p_{h}(x,s)(z_{h,u})_x(x)dx
    +\hat{p}(x_{j+\frac{1}{2}},s)z_{h,u}(x_{j+\frac{1}{2}}^{-})
    -\hat{p}(x_{j-\frac{1}{2}},s)z_{h,u}(x_{j-\frac{1}{2}}^{+}) \\
    & \quad\quad\quad\quad + \int_{I_{j}} \left[ \lambda v_{h} + \mu \psi_{h} + \Gamma(\cdot,u_{h},v_{h},\psi_{h}) \right](x,s) z_{h,u}(x)dx\Bigg\}ds - \int_{t}^{T}\int_{I_{j}}\psi_{h}(x,s)z_{h,u}(x)dx dW_{s}.
    \end{aligned}
\end{equation*}
From the above formula and the properties of continuous semimartingale, we can get that for any $z_{h,u}\in V_{h}^k$ and continuous semimartingale $Y$,
\begin{equation}\nonumber
    \int_{I_{j}}z_{h,u}(x)\left<u_h(x,\cdot),Y\right>_{t}^{T}dx
    =\Big<\int_{I_{j}}z_{h,u}(x)u_{h}(x,\cdot)dx,Y\Big>_{t}^{T}
    =\Big<
    -\int_{\cdot}^{T}\int_{I_{j}}\psi_{h}(x,s)z_{h,u}(x)dxdW_{s},Y
    \Big>_{t}^{T}.
\end{equation}
Then we have
    \begin{align}
    &\int_{I_{j}}\left<u_{h}(x,\cdot),u_{h}(x,\cdot)\right>_{t}^{T}dx
    =\int_{I_{j}}\Big<u_{h}(x,\cdot),
    \sum_{l=0}^{k}\boldsymbol{u}_{h}^{l,j}(\cdot)\phi_{l}^{j}(x)\Big>_{t}^{T}dx
    =\sum_{l=0}^{k}\int_{I_{j}}\phi_{l}^{j}(x)
    \Big<u_{h}(x,\cdot),\boldsymbol{u}_{h}^{l,j}(\cdot)\Big>_{t}^{T}dx \nonumber\\
    =&\sum_{l=0}^{k}
    \Big<-\int_{\cdot}^{T}
    \int_{I_{j}}\psi_{h}(x,s)\phi_{l}^{j}(x)dxdW_{s},
    \boldsymbol{u}_{h}^{l,j}(\cdot)\Big>_{t}^{T} = \int_{I_{j}} \int_{t}^{T}
    \sum_{n=0}^{k}\boldsymbol{\psi}_{h}^{n,j}(s)
    \phi_{n}^{j}(x)d\left<W,u_h(x,\cdot)\right>_{s}dx \nonumber\\
    =& - \int_{I_{j}}\Big< \int_{\cdot}^{T} \sum_{n=0}^{k} \boldsymbol{\psi}_{h}^{n,j}(s) \phi_{n}^{j}(x) dW_{s},u_{h}(x,\cdot)\Big>_{t}^{T}dx = -\sum_{n=0}^{k}\int_{I_{j}}
    \phi_{n}^{j}(x)\Big<
    \int_{\cdot}^{T}\boldsymbol{\psi}_{h}^{n,j}(s)dW_{s},
    u_{h}(x,\cdot)\Big>_{t}^{T}dx \nonumber\\
    =&\sum_{n=0}^{k}\Big<
    \int_{\cdot}^{T}\boldsymbol{\psi}_{h}^{n,j}(s)dW_{s},
    \int_{\cdot}^{T}\int_{I_{j}}\sum_{m=0}^{k}\boldsymbol{\psi}_{h}^{m,j}(s)
    \phi_{m}^{j}(x)\phi_{n}^{j}(x)dW_{s}dx
    \Big>_{t}^{T}\nonumber \\
    =&\int_{I_{j}}\int_{t}^{T}
    \sum_{n=0}^{k}\boldsymbol{\psi}_{h}^{n,j}(s)\phi_{n}^{j}(x)
    \sum_{m=0}^{k}\boldsymbol{\psi}_{h}^{m,j}(s)\phi_{m}^{j}(x)
    dsdx =\int_{I_{j}}\int_{t}^{T}|\psi_{h}(x,s)|^{2}dsdx.\label{entropy_for_uh_estimate}
    \end{align}
Sum the indicators $j$ in the above formula from $1$ to $N$, and take expectations on both sides, we get
\begin{equation}\nonumber
    \mathcal{T}_{1}(t)
    =-\mathbb{E}\Big[\int_{0}^{b}
    \left<u_{h}(x,\cdot),u_{h}(x,\cdot)
    \right>_{t}^{T} dx\Big]
    =-\mathbb{E}\Big[\int_{0}^{b}
    \int_{t}^{T}|\psi_{h}(x,s)|^{2}dsdx\Big]
    =-\mathbb{E}\Big[\int_{t}^{T}
    \|\psi_{h}(\cdot,s)\|^{2}ds\Big].
\end{equation}

\textbf{Estimation of $\mathcal{T}_{2}(t)$.}

From the definition \eqref{def_lambda} of $\lambda(x,t)$ and assumption $(\mathcal{A}_{2})$, it is known that for any $(x,t)\in [0,T]\times[0,b]$, we have
$|\lambda(x,t)|\leq 2K^{2}$. Then, for any $\epsilon>0$,
\begin{equation}\nonumber
\int_{0}^{b}\lambda(x,t)v_{h}(x,t)u_{h}(x,t)dx
\leq {\epsilon} \|v_{h}(\cdot,t)\|^2 +\frac{K^{4}}{4\epsilon}\|u_{h}(\cdot,t)\|^2.
\end{equation}
From the definition \eqref{def_mu} of $\mu(x,t)$ and assumption $(\mathcal{A}_{2})$, it is known that for any $(x,t)\in [0,T]\times[0,b]$, we have
$|\mu(x,t)|\leq K $. Then, for any $\epsilon > 0$, 
\begin{equation}\nonumber
\int_{0}^{b} \mu(x,t) \psi_{h}(x,t) u_{h}(x,t)dx
\leq \epsilon \|\psi_{h}(\cdot,t)\|^2 + \frac{K^{2}}{4\epsilon}\|u_{h}(\cdot,t)\|^2.
\end{equation}
Due to assumptions $(\mathcal{A}_3)$, it holds
\begin{align}
    &\int_{0}^{b}\Gamma(x,t,u_{h},v_{h},\psi_{h})
u_{h}(x,t)dx
    \leq \int_{0}^{b}
    |\Gamma(x,t,u_{h},v_{h},\psi_{h})-\Gamma^{0}(x,t)| \cdot |u_{h}(x,t)|dx
    +\int_{0}^{b}|\Gamma^{0}(x,t)|\cdot |u_{h}(x,t)|dx \nonumber\\
    &\leq \frac{\epsilon}{L^2}
    \|\Gamma(x,t,u_{h},v_{h},\psi_{h})-\Gamma^{0}(x,t)\|^{2}
    +\frac{L^2}{4\epsilon}
    \|u_{h}(\cdot,t)\|^{2}
    +\int_{0}^{b}|\Gamma^{0}(x,t)|\cdot|u_{h}(x,t)|dx \nonumber\\
    &\leq \epsilon
    \left(\|u_{h}(\cdot,t)\|^{2}+\|v_{h}(\cdot,t)\|^{2}
    +\|\psi_{h}(\cdot,t)\|^{2}\right)
    +\frac{L^2}{4\epsilon}
    \|u_{h}(\cdot,t)\|^{2}
    + \frac{1}{2} \|\Gamma^{0}(\cdot,t)\|^2 + \frac{1}{2}\|u_{h}(\cdot,t)\|^2.\nonumber
\end{align}
We have
\begin{equation*}
    \mathcal{T}_{2}(t) \leq \mathbb{E}\left[
        \int_{t}^{T} \left( \frac{4K^4+K^2+L^2}{2\epsilon}+2\epsilon+ 1 \right) \|u_{h}(\cdot,s)\|^2 + 4\epsilon \|v_{h}(\cdot,s)\|^2 + 4\epsilon \|\psi_{h}(\cdot,s)\|^2 + \|\Gamma^{0}(\cdot,s)\|^2 ds\right].
\end{equation*}

\textbf{Estimation of $\mathcal{T}_{3}(t)$.}

It follows \eqref{bound_for_element_u_psi} that
\begin{equation*}
    \mathbb{E}\Big[
    \mathop{\text{sup}}_{0\leq t \leq T}
    \|u_{h}(\cdot,t)\|^{2}\Big]
    +\mathbb{E}\Big[
    \int_{0}^{T}
    \|\psi_{h}(\cdot,t)\|^{2}dt\Big]
    < \infty.
\end{equation*}
Using Cauchy-Schwartz's inequality, we obtain
    \begin{align*}
    &\mathbb{E}\left[
    \left( \int_{0}^{T} \left( \int_{0}^{b}
    \psi_{h}(x,t)u_{h}(x,t)dx \right)^{2}dt \right)^{\frac{1}{2}} \right]
    \leq \mathbb{E}\left[
    \left( \int_{0}^{T}\|u_{h}(\cdot,t)\|^{2}\|\psi_{h}(\cdot,t)\|^{2}dt \right)^{\frac{1}{2}}
    \right]\\
    \leq& \mathbb{E} \left[
    \mathop{\text{sup}}_{0\leq t\leq T}\|u_{h}(\cdot,t)\|
    \left( \int_{0}^{T}\|\psi_{h}(\cdot,t)\|^{2}dt \right)^{\frac{1}{2}} \right] \leq \left( \mathbb{E} \left[
    \mathop{\text{sup}}_{0\leq t\leq T}
    \|u_{h}(\cdot,t)\|^{2}\right]\right)^{\frac{1}{2}}
    \left(\mathbb{E}\left[\int_{0}^{T}
    \|\psi_{h}(\cdot,t)\|^{2}dt\right]\right)^{\frac{1}{2}}
    < \infty.
    \end{align*}
Thus $\int_0^\cdot \int_{0}^{b}\psi_{h}(x,t)u_{h}(x,t)dx dW_t$ is a integrable martingale, then thanks to the properties of martingale, we deduce 
\begin{equation}\nonumber
    \mathcal{T}_{3}(t) \equiv 0,\quad t\in[0,T].
\end{equation}

\textbf{Estimation of $\mathcal{T}_{4}(t)$.}

Note that
\begin{align}
   &\sum_{j=1}^{N}\Bigg[ -\int_{I_{j}}u_{h}(x,t)(p_h)_x(x,t)dx +\hat{u}(x_{j+\frac{1}{2}},t)p_h(x_{j+\frac{1}{2}}^{-},t)-\hat{u}(x_{j-\frac{1}{2}},t)p_h(x_{j-\frac{1}{2}}^{+},t) \nonumber\\
   &\quad\quad - \int_{I_{j}}p_{h}(x,t)(u_h)_x(x,t)dx
        +\hat{p}(x_{j+\frac{1}{2}},t)u_h(x_{j+\frac{1}{2}}^{-},t)
        -\hat{p}(x_{j-\frac{1}{2}},t)u_h(x_{j-\frac{1}{2}}^{+},t)\Bigg]\nonumber\\
    =&\sum_{j=2}^{N-1}\Bigg[
    -\int_{I_{j}}u_h(x,t)(p_h)_{x}(x,t)dx
    +u_h\big(x_{j+\frac{1}{2}}^{-},t\big)
    p_h\big(x_{j+\frac{1}{2}}^{-},t\big)
    -u_h\big(x_{j-\frac{1}{2}}^{-},t\big)
    p_h\big(x_{j-\frac{1}{2}}^{+},t\big) \nonumber\\
    &\quad\quad -\int_{I_{j}}p_h(x,t)(u_h)_{x}(x,t)dx
    +p_h\big(x_{j+\frac{1}{2}}^{+},t\big)
    u_h\big(x_{j+\frac{1}{2}}^{-},t\big)
    -p_h\big(x_{j-\frac{1}{2}}^{+},t\big)
    u_h\big(x_{j-\frac{1}{2}}^{+},t\big)\Bigg]\nonumber\\
    & -\int_{I_{1}}u_h(x,t)(p_h)_{x}(x,t)dx
    +u_h\big(x_{\frac{3}{2}}^{-},t\big)
    p_h\big(x_{\frac{3}{2}}^{-},t\big)
    -u_h\big(x_{\frac{1}{2}}^{+},t\big)
    p_h\big(x_{\frac{1}{2}}^{+},t\big) -\int_{I_{1}}p_h(x,t)(u_h)_{x}(x,t)dx
    +p_h\big(x_{\frac{3}{2}}^{+},t\big)
    u_h\big(x_{\frac{3}{2}}^{-},t\big)\nonumber\\
    & -\int_{I_{N}} \left[u_h(p_h)_{x} + p_h(u_h)_{x} \right](x,t)dx
    +u_h\big(x_{N+\frac{1}{2}}^{-},t\big)
    p_h\big(x_{N+\frac{1}{2}}^{-},t\big)
    -u_h\big(x_{N-\frac{1}{2}}^{-},t\big)
    p_h\big(x_{N-\frac{1}{2}}^{+},t\big) 
    -p_h\big(x_{N-\frac{1}{2}}^{+},t\big)
    u_h\big(x_{N-\frac{1}{2}}^{+},t\big) \nonumber\\
    =& p_h\big(x_{N-\frac{1}{2}}^{+},t\big)
    u_h\big(x_{N-\frac{1}{2}}^{-},t\big) - p_h\big(x_{\frac{3}{2}}^{+},t\big)
    u_h\big(x_{\frac{3}{2}}^{-},t\big) + p_h\big(x_{\frac{3}{2}}^{+},t\big)
    u_h\big(x_{\frac{3}{2}}^{-},t\big) - 
    u_h\big(x_{N-\frac{1}{2}}^{-},t\big) p_h\big(x_{N-\frac{1}{2}}^{+},t\big)=0.\label{eq2408081}
\end{align}

Moreover, it can be known from the approximation equation \eqref{approx_eq_c} that
    \begin{align}
    &\int_{0}^{b}v_{h}(x,t)p_{h}(x,t)dx
    =\int_{0}^{b}\left[\frac{1}{2}\left(\sigma^{2}(x,t)
    +\bar{\sigma}^{2}(x,t)\right)v_{h}(x,t)
    +\sigma(x,t)\psi_{h}(x,t)\right]v_{h}(x,t)dx \nonumber\\
    \geqslant& \int_{0}^{b}
    \frac{1}{2}\left(\sigma^{2}(x,t)
    +\bar{\sigma}^{2}(x,t)\right) |v_{h}(x,t)|^{2}dx
    -\frac{1+\frac{\kappa^2}{2K^2}}{2}
    \int_{0}^{b} \sigma^{2}(x,t) |v_{h}(x,t)|^{2}dx
    -\frac{1}{2+ \frac{\kappa^2}{K^2} }\int_{0}^{b}|\psi_{h}(x,t)|^{2}dx \nonumber\\
    \geqslant& 
    \frac{\kappa^2}{4} \|v_{h}(\cdot,t)\|^{2} 
    -\frac{1}{2}\left( 1-\frac{\kappa^2}{2K^2+\kappa^2}\right) \|\psi_{h}(\cdot,t)\|^{2}.\label{stability_step1_T_0}
    \end{align}

Concluding the above estimates, we obtain
\begin{align*}
    &\mathbb{E}\left[\|u_{h}(\cdot,s)\|^2\right] + \left( \frac{\kappa^2}{2} - 4\epsilon \right) \mathbb{E}\left[ \int_{t}^{T}  \|v_{h}(\cdot,s)\|^2 ds\right] + \left( \frac{\kappa^2}{2K^2+\kappa^2} - 4\epsilon \right) \mathbb{E}\left[ \int_{t}^{T}  \|\psi_{h}(\cdot,s)\|^2 ds\right]\\
    \leq& \left( \frac{4K^4+K^2+L^2}{2\epsilon}+2\epsilon+ 1 \right) \mathbb{E}\left[ \int_{t}^{T} \|u_{h}(\cdot,s)\|^2 ds\right] + \mathbb{E}\left[ \int_{t}^{T} \|\Gamma^{0}(\cdot,s)\|^2 ds\right] + \mathbb{E}\left[ \|G\|^2 \right].
\end{align*}
Taking $\epsilon$ small enough, we have
\begin{equation*}
        \mathbb{E}\Big[\|u_{h}(\cdot,t)\|^{2}\Big]
        +\mathbb{E}\Big[
        \int_{t}^{T}\int_{0}^{b}\left(
        |v_{h}|^{2}
        +|\psi_{h}|^{2}\right)(x,s)dxds\Big] \leq \mathbb{E}\Big[\|G\|^{2}\Big] 
        +C\Big(\mathbb{E}\Big[
        \int_{t}^{T}\|u_{h}(\cdot,s)\|^{2}ds\Big]
        +\mathbb{E}\Big[
        \int_{t}^{T} \|\Gamma^{0}(\cdot,s)\|^2 ds
        \Big]\Big).
\end{equation*}
Furthermore, combining Gronwall's inequality, we can get
\begin{equation}\nonumber
    \mathbb{E}\left[\|u_{h}(\cdot,t)\|^{2}\right]
    \leq \mathbb{E}\Big[\|G\|^{2}\Big]
    +C\mathbb{E}\Big[
    \int_{0}^{T} \|\Gamma^{0}(\cdot,s)\|^2 ds \Big],
\end{equation}
It follows that
\begin{equation}\label{stability_step1_conclusion}
    \begin{aligned}
    \mathop{\text{sup}}_{0\leq t \leq T}
    \mathbb{E}\Big[\|u_{h}(\cdot,t)^{2}|\Big]
    +\mathbb{E}\Big[
    \int_{0}^{T}\|v_{h}(\cdot,s)\|^{2}ds\Big]
    +\mathbb{E}\Big[
    \int_{0}^{T}\|\psi_{h}(\cdot,s)\|^{2}ds
    \Big]
    \leq C\left(\mathbb{E}\Big[\|G\|^{2}\Big]
    +\mathbb{E}\Big[
    \int_{0}^{T} \|\Gamma^{0}(\cdot,s)\|^2 ds
    \Big]\right).
    \end{aligned}
\end{equation}
Step 2. In view of~\eqref{stability_equation}, we get
\begin{equation*}
    \begin{aligned}
    \mathbb{E}\Big[
    \mathop{\text{sup}}_{t\leq r\leq T}\
    |u_{h}(\cdot,r)\| ^{2}\Big]
    \leq
    \mathbb{E}\Big[\|G(\cdot)\|^{2}\Big]
    +\mathcal{T}_{5}(t)+\mathcal{T}_{6}(t)
    +\mathcal{T}_{7}(t)+\mathcal{T}_{8}(t),
    \end{aligned}
\end{equation*}
with
    \begin{align*}
    \mathcal{T}_{5}(t):=& 2\mathbb{E}\Big[
    \mathop{\text{sup}}_{t\leq r\leq T}
    \int_{r}^{T}\int_{0}^{b}
    -\sigma(x,s)\psi_{h}(x,s)v_{h}(x,s)dxds\Big],\\
    \quad\mathcal{T}_{6}(t):=&
    2\mathbb{E}\Big[
    \mathop{\text{sup}}_{t\leq r\leq T}
    \int_{r}^{T}\int_{0}^{b}\left[\lambda v_{h}u_{h} + \mu\psi_{h}u_{h} + \Gamma(\cdot,u_{h},v_{h},\psi_{h}) u_{h} \right](x,s) dxds\Big],\\
    \mathcal{T}_{7}(t)
    :=&2\mathbb{E}\Big[
    \mathop{\text{sup}}_{t\leq r\leq T}
    \int_{r}^{T}\int_{0}^{b}
    -\psi_{h}(x,s)u_{h}(x,s)dxdW_{s}\Big],\\
    \mathcal{T}_{8}(t)
    :=&2\mathbb{E}\Bigg[
    \sup_{t\leq r \leq T} \int_{r}^{T}\sum_{j=1}^{N}\bigg\{
    -\int_{I_{j}}p_{h}(x,s)(u_h)_x(x,s)dx
        +\hat{p}(x_{j+\frac{1}{2}},s)u_{h}(x_{j+\frac{1}{2}}^{-},s)
        -\hat{p}(x_{j-\frac{1}{2}},s)u_{h}(x_{j-\frac{1}{2}}^{+},s) \\
    & \quad\quad\quad\quad\quad\quad\quad\quad -\int_{I_{j}}u_{h}(x,s)(p_{h})_x(x,s)dx +\hat{u}(x_{j+\frac{1}{2}},s)p_{h}(x_{j+\frac{1}{2}}^{-},s)-\hat{u}(x_{j-\frac{1}{2}},s)p_{h}(x_{j-\frac{1}{2}}^{+},s) \bigg\}ds \Bigg].
    \end{align*}

\textbf{Estimation of $\mathcal{T}_{5}(t)$.}

By assumption $(\mathcal{A}_{2})$, one has
    \begin{align*}
    \mathop{\text{sup}}_{t\leq r\leq T}
    \int_{r}^{T}\int_{0}^{b}
    -\sigma(x,s)\psi_{h}(x,s)v_{h}(x,s)dxds
    &\leq
    \int_{t}^{T}\int_{0}^{b}
    |\sigma(x,s)\psi_{h}(x,s)v_{h}(x,s)|dxds\\
    &\leq
    \frac{1}{2}\int_{t}^{T}\int_{0}^{b}
    K^{2}|\psi_{h}(x,s)|^{2}dxds
    +\frac{1}{2}\int_{t}^{T}\int_{0}^{b}
    |v_{h}(x,s)|^{2}dxds.
    \end{align*}
Hence,
\begin{equation}\nonumber
    \mathcal{T}_{5}(t)\leq 
    C\Big(\mathbb{E}\Big[
    \int_{t}^{T}\int_{0}^{b}
    |\psi_{h}(x,s)|^{2}dxds\Big]
    +\mathbb{E}\Big[
    \int_{t}^{T}\int_{0}^{b}
    |v_{h}(x,s)|^{2}dxds\Big]\Big).
\end{equation}

\textbf{Estimation of $\mathcal{T}_{6}(t)$.}

It holds that
\begin{equation*}
    \begin{aligned}
    \mathop{\text{sup}}_{t\leq r\leq T}
    \int_{r}^{T}\int_{0}^{b} \left[\lambda v_{h}u_{h} + \mu\psi_{h}u_{h} + \Gamma(\cdot,u_{h},v_{h},\psi_{h}) u_{h} \right](x,s) dxds \leq\int_{t}^{T}\int_{0}^{b}\big| \lambda v_{h}u_{h} + \mu\psi_{h}u_{h} + \Gamma(\cdot,u_{h},v_{h},\psi_{h}) u_{h} \big|(x,s) dxds.
    \end{aligned}
\end{equation*}
Considering the first two terms in $\mathcal{T}_{6}(t)$, we have
    \begin{align}
    &\int_{t}^{T}\int_{0}^{b}
    |\lambda(x,s)v_{h}(x,s)u_{h}(x,s)|dxds + \int_{t}^{T}
    \int_{0}^{b}|\mu(x,s)\psi_{h}(x,s)u_{h}(x,s)|dxds \nonumber\\
    \leq &
    C\int_{t}^{T}\|v_{h}(\cdot,s)\|^{2} + \|\psi_{h}(\cdot,s)\|^{2} + \|u_{h}(\cdot,s)\|^{2}ds\leq
    C\int_{t}^{T}\|v_{h}(\cdot,s)\|^{2} + \|\psi_{h}(\cdot,s)\|^{2} ds
    +C\int_{t}^{T}
    \mathop{\text{sup}}_{s\leq r\leq T}
    \|u_{h}(\cdot,r)\|^{2}ds.\nonumber
    \end{align}
We have for the third term of $\mathcal{T}_{6}(t)$, 
    \begin{align}
    &\int_{t}^{T}\int_{0}^{b}
    \left|\Gamma(\cdot,u_{h},v_{h},\psi_{h})
    u_{h}\right|(x,s)dxds
    \leq
    C\int_{t}^{T}\left(
    \|u_{h}(\cdot,s)\|^{2}
    +\|v_{h}(\cdot,s)\|^{2}
    +\|\psi_{h}(\cdot,s)\|^{2}
    +\|\Gamma^{0}(\cdot,s)\|^{2}\right)ds \nonumber\\
    &\leq
    C\int_{t}^{T}\left(
    \|v_{h}(\cdot,s)\|^{2}
    +\|\psi_{h}(\cdot,s)\|^{2}
    +\|\Gamma^{0}(\cdot,s)\|^{2}\right)ds
    +C\int_{t}^{T}
    \mathop{\text{sup}}_{s\leq r\leq T}
    \|u_{h}(\cdot,r)\|^{2}ds.\nonumber
    \end{align}

\textbf{Estimation of $\mathcal{T}_{7}(t)$.}

Using the BDG inequality and Cauchy-Schwartz's inequality, we obtain
    \begin{align*}
    &\mathbb{E}\Big[
    \big|\mathop{\text{sup}}_{t\leq r\leq T}
    \int_{r}^{T}\int_{0}^{b}
    (\psi_{h} u_{h})(x,s)dxdW_{s}\big|\Big]
    \leq 
    C\mathbb{E}\Big[\big(
    \int_{t}^{T}\big|\int_{0}^{b}
    (\psi_{h} u_{h})(x,s)dx\big|^{2}ds
    \big)^{\frac{1}{2}}\Big] \leq C\mathbb{E}\Big[\big(
    \int_{t}^{T}\|u_{h}(\cdot,s)\|^{2}
    \|\psi_{h}(\cdot,s)\|^{2}ds
    \big)^{\frac{1}{2}}\Big]\\
    \leq& C\mathbb{E}\Big[
    \mathop{\text{sup}}_{t\leq r\leq T}
    \|u_{h}(\cdot,r)\|
    \big(\int_{t}^{T}
    \|\psi_{h}(\cdot,s)\|^{2}ds
    \big)^{\frac{1}{2}}\Big]\leq
    \frac{1}{4}\mathbb{E}\Big[
    \mathop{\text{sup}}_{t\leq r\leq T}
    \|u_{h}(\cdot,r)\|^{2}
    \Big]
    +C\mathbb{E}\Big[
    \int_{t}^{T}
    \|\psi_{h}(\cdot,s)\|^{2}ds\Big],
    \end{align*}
which gives that
\begin{equation}\nonumber
    \mathcal{T}_{7}(t)
    \leq
    \frac{1}{2}\mathbb{E}\Big[
    \mathop{\text{sup}}_{t\leq r\leq T}
    \|u_{h}(\cdot,r)\|^{2}
    \Big]
    +C\mathbb{E}\Big[
    \int_{t}^{T}
    \|\psi_{h}(\cdot,s)\|^{2}ds\Big].
\end{equation}

\textbf{Estimation of $\mathcal{T}_{8}(t)$.}

From~\eqref{eq2408081}, it is known that
\begin{equation}\nonumber
    \mathcal{T}_{8}(t)=0,\quad t\in[0,T].
\end{equation}

Concluding the above estimates, we have
\begin{equation*}
    \begin{aligned}
    \mathbb{E}\Big[
    \mathop{\text{sup}}_{t\leq r\leq T}\
    \|u_{h}(\cdot,r)\| ^{2}\Big]
    \leq \Big(\mathbb{E}\Big[\|G\|^{2}\Big] + C\mathbb{E}\Big[\int_{t}^{T} \|v_{h}(\cdot,s)\|^{2} + \|\psi_{h}(\cdot,s)\|^{2}  + \|\Gamma^{0}(\cdot,s)\|^{2}ds\Big]
    +C\int_{t}^{T}
    \mathbb{E}\big[
    \mathop{\text{sup}}_{s\leq r\leq T}
    \|u_{h}(\cdot,r)\|^{2}\big]ds.
    \end{aligned}
\end{equation*}
From the conclusion \eqref{stability_step1_conclusion} in the first step and Gronwall's inequality, we get
\begin{equation}\label{stability_step2_conclusion_pre}
    \mathbb{E}\Big[
    \mathop{\text{sup}}_{t\leq r\leq T}\
    \|u_{h}(\cdot,r)\| ^{2}\Big]
    \leq
    C\Big(\mathbb{E}\Big[\|G\|^{2}\Big]
    +\mathbb{E}\Big[
    \int_{t}^{T}
    \|\Gamma^{0}(\cdot,s)\|^{2}ds\Big]\Big),
\end{equation}
Combining the first step conclusion \eqref{stability_step1_conclusion} and the second step conclusion \eqref{stability_step2_conclusion_pre}, the stability of the LDG method can be obtained.  
\end{proof}

\subsection{The optimal estimate error of the LDG method}
In order to establish the optimal estimate error of our LDG method for BSPDE \eqref{BSPDE_in_Qiu_paper}, we begin with the following estimate: 
    \begin{equation}
   \mathop{\text{sup}}_{0\leq t\leq T}
    \mathbb{E}\Big[\|u(\cdot,t)-u_{h}(\cdot,t)\|^{2}\Big]^{\frac{1}{2}}
    +\mathbb{E}\Big[\int_{0}^{T}
    \|u_{x}(\cdot,s)-v_{h}(\cdot,s)\|^{2}ds\Big]^{\frac{1}{2}}
    +\mathbb{E}\Big[\int_{0}^{T}
    \|\psi(\cdot,s)-\psi_{h}(\cdot,s)\|^{2}ds\Big]^{\frac{1}{2}}
    \leq Ch^{k+1}.
    \end{equation}

To prove this, we note that in the approximation equation \eqref{approx_eq}, the numerical solution $(u_{h},v_{h},\psi_{h})$ is replaced by the real solution $(u,u_x,\psi)$, the original  numerical format still holds. This is because that Lemma \ref{highordersolution} ensures that the real solution $(u,u_x,\psi)$ is continuous w.r.t. $x$, then if bring it into the approximation equation, the numerical flux is exactly the values at the endpoints, which are formally consistent with the equation that the numerical solution satisfies.

For $\varphi=u,v,\psi,p$, we define the error terms as follows
\begin{equation}\label{error_step1_error_def}
  e_{\varphi}(x,t):=\left(\varphi-\varphi_{h}\right)(x,t)
  =\left(\xi_{\varphi}-\eta_{\varphi}\right)(x,t),
\end{equation}
with 
  \begin{align*}
  &\xi_{u}(x,t):=\left(\mathcal{P}^{-}u-u_{h}\right)(x,t),\quad\quad
  \eta_{u}(x,t):=\left(\mathcal{P}^{-}u-u\right)(x,t),\\
  &\xi_{v}(x,t):=\left(\mathcal{P}^{+}v-v_{h}\right)(x,t),\quad\quad 
  \eta_{v}(x,t):=\left(\mathcal{P}^{+}v-v\right)(x,t),\\
  &\xi_{\psi}(x,t):=\left(\mathcal{P}^{-}\psi-\psi_{h}\right)(x,t),\quad\quad 
  \eta_{\psi}(x,t):=\left(\mathcal{P}^{-}\psi-\psi\right)(x,t),\\
  &\xi_{p}(x,t):=\left(\mathcal{P}^{+}p-p_{h}\right)(x,t),\quad\quad 
  \eta_{p}(x,t):=\left(\mathcal{P}^{+}p-p\right)(x,t).
  \end{align*}
where $\mathcal{P}^{-}$ and $\mathcal{P}^{+}$ are Gauss-Radau projections defined in Definition \ref{Gauss_Radou_projection_plus} and \ref{Gauss_Radou_projection_minus}.

We can get the error equation as follows
\begin{subnumcases}{}\nonumber
        \int_{I_{j}}d e_u(x,t)z_{h,u}(x)dx=-\Bigg\{ -\int_{I_{j}}e_p(x,t)(z_{h,u})_x(x)dx
        +\hat{e}_p(x_{j+\frac{1}{2}},t)z_{h,u}(x_{j+\frac{1}{2}}^{-})
        -\hat{e}_p(x_{j-\frac{1}{2}},t)z_{h,u}(x_{j-\frac{1}{2}}^{+})\nonumber\\
        \quad\quad\quad\quad\quad\quad\quad\quad +\int_{I_{j}}\left[\lambda e_v + \mu e_\psi + \left(\Gamma(\cdot,u,v,\psi) - \Gamma(\cdot,u_{h},v_{h},\psi_{h})\right) \right](x,t) z_{h,u}(x)dx \Bigg\}dt +\int_{I_{j}} e_\psi(x,t) z_{h,u}(x)dx dW_{t},\nonumber\\
        \int_{I_{j}} e_v(x,t)z_{h,v}(x)dx=-\int_{I_{j}} e_u(x,t) (z_{h,v})_x(x)dx +\hat{e}_u(x_{j+\frac{1}{2}},t)z_{h,v}(x_{j+\frac{1}{2}}^{-})-\hat{e}_u(x_{j-\frac{1}{2}},t)z_{h,v}(x_{j-\frac{1}{2}}^{+}),\nonumber\\
        \int_{I_{j}} e_p(x,t)z_{h,p}(x)dx=\int_{I_{j}}\left[\frac{1}{2}\left(\sigma^{2}(x,t)+\bar{\sigma}^{2}(x,t)\right) e_v(x,t) +\sigma(x,t) e_\psi(x,t) \right] z_{h,p}(x)dx,\nonumber\\
        \int_{I_{j}} e_u(x,T)z_{h,G}(x)dx= 0.\nonumber
\end{subnumcases}

In view the definition of the Gauss-Radau projection $\mathcal{P}^{+}$, we have
\begin{align}
    & -\int_{I_{j}}d \xi_u(x,t)z_{h,u}(x)dx + \int_{I_{j}}d \eta_u(x,t)z_{h,u}(x)dx + \int_{I_{j}} e_\psi(x,t) z_{h,u}(x)dx dW_{t} \nonumber\\
    =& \Bigg\{ -\int_{I_{j}} \xi_p(x,t)(z_{h,u})_x(x)dx
        +\hat{\xi}_p(x_{j+\frac{1}{2}},t)z_{h,u}(x_{j+\frac{1}{2}}^{-})
        -\hat{\xi}_p(x_{j-\frac{1}{2}},t)z_{h,u}(x_{j-\frac{1}{2}}^{+})\nonumber\\
    & \quad\quad +\int_{I_{j}}\left[\lambda (\xi_v -\eta_v) + \mu (\xi_\psi - \eta_\psi) + \left(\Gamma(\cdot,u,v,\psi) - \Gamma(\cdot,u_{h},v_{h},\psi_{h})\right) \right](x,t) z_{h,u}(x)dx \Bigg\}dt,\nonumber
\end{align}
where
$$
\hat{\xi}_p(x_{j+\frac{1}{2}},t):=\xi_{p}(x_{j+\frac{1}{2}}^{+},t),
    \quad j=1,2...,N-1,\quad\quad\text{and}\quad\quad \hat{\xi}_p(x_{\frac{1}{2}},t)=\hat{\xi}_p(x_{N+\frac{1}{2}},t)=0.
$$
Similarly, by the definition of the Gauss-Radau projection $\mathcal{P}^{-}$, taking $z_{h,v}=\xi_p(\cdot,t)$, we get
$$
\int_{I_{j}} (\xi_v \xi_p)(x,t)dx=\int_{I_{j}} (\eta_v \xi_p)(x,t) dx - \int_{I_{j}} \left[ \xi_u (\xi_p)_x\right](x,t) dx +\hat{\xi}_u(x_{j+\frac{1}{2}},t) \xi_p(x_{j+\frac{1}{2}}^{-},t)-\hat{\xi}_u(x_{j-\frac{1}{2}},t) \xi_p(x_{j-\frac{1}{2}}^{+},t),
$$
where
$$
    \hat{\xi}_u(x_{j+\frac{1}{2}},t):=\xi_{u}(x_{j+\frac{1}{2}}^{-},t),
    \quad j=1,2...,N-1,\quad\quad\text{and}\quad\quad
    \hat{\xi}_u(x_{\frac{1}{2}},t)=\xi_{u}(0^{+},t),\quad 
    \hat{\xi}_u(x_{N+\frac{1}{2}},t)=\xi_{u}(b^{-},t),
$$
in which we use the fact that
$$
\xi_p(0^+,t)= \mathcal{P}^{+}p (0^+,t) - p_{h}(0^+,t) = p (0^+,t) - 0 = 0.
$$
By the analogical computation that leads to~\eqref{eq2408081}, we get
\begin{align*}
    \sum_{j=1}^N \Bigg\{ &-\int_{I_{j}} \left[ \xi_u (\xi_p)_x\right](x,t) dx +\hat{\xi}_u(x_{j+\frac{1}{2}},t) \xi_p(x_{j+\frac{1}{2}}^{-},t)-\hat{\xi}_u(x_{j-\frac{1}{2}},t) \xi_p(x_{j-\frac{1}{2}}^{+},t)\\
    & -\int_{I_{j}} \left[ \xi_p (\xi_u)_x\right](x,t) dx +\hat{\xi}_p(x_{j+\frac{1}{2}},t) \xi_u(x_{j+\frac{1}{2}}^{-},t)-\hat{\xi}_p(x_{j-\frac{1}{2}},t) \xi_u(x_{j-\frac{1}{2}}^{+},t) \Bigg\} \equiv 0.
\end{align*}
Taking $z_{h,u}=\xi_u(\cdot,t)$ and $z_{h,p}=\xi_v(\cdot,t)$, as well as applying It\^o's formula to $|\xi_{u}(x,t)|^2$, it follows that
\begin{align}
    & \|\xi_u(\cdot,t)\|^2 + \int_{0}^{b}\left< \xi_{u}(x,\cdot),\xi_{u}(x,\cdot) \right>_{t}^{T} dx + \int_t^T \int_0^b \left[\left(\sigma^{2}+\bar{\sigma}^{2}\right) |\xi_v|^2 +2\sigma \xi_\psi \xi_v \right] (x,s) dx ds \nonumber\\
    =& \|\xi_u(\cdot,T)\|^2 - 2 \int_0^b \int_t^T \xi_u(x,s) d \eta_u(x,s)dx - 2 \int_t^T \int_0^b e_\psi(x,s) \xi_u(x,s) dx dW_s + 2 \int_t^T \int_0^b \eta_v(x,s) \xi_p(x,s) dxds \nonumber\\
    & - 2 \int_t^T \int_0^b \eta_p(x,s) \xi_v(x,s) dxds + \int_t^T \int_0^b \left[\left(\sigma^{2}+\bar{\sigma}^{2}\right) \eta_v \xi_v +2\sigma \eta_\psi \xi_v \right] (x,s) dx ds \nonumber\\
    & + 2 \int_t^T \int_0^b \left[\lambda (\xi_v -\eta_v)\xi_u + \mu (\xi_\psi - \eta_\psi)\xi_u + \left(\Gamma(\cdot,u,v,\psi) - \Gamma(\cdot,u_{h},v_{h},\psi_{h})\right)\xi_u \right](x,s) dx ds,\label{error_step1_error_equation_plugin_0}
\end{align}
which yields that
\begin{equation*}
  \begin{aligned}
  \mathbb{E}\left[ \|\xi_u(\cdot,t)\|^2 \|^{2}\right]
  +\mathcal{S}_{0}(t)+\mathcal{S}_{1}(t)
  =\mathbb{E}\left[\|\xi_{u}(\cdot,T)\|^{2}\right]
  +\mathcal{S}_{2}(t)
  +\mathcal{S}_{3}(t)
+\mathcal{S}_{4}(t)
  +\mathcal{S}_{5}(t),
  \end{aligned}
\end{equation*}
with
  \begin{align*}
  &\mathcal{S}_{0}(t)= \mathbb{E} \left[ \int_t^T \int_0^b \left[\left(\sigma^{2}+\bar{\sigma}^{2}\right) |\xi_v|^2 +2\sigma \xi_\psi \xi_v \right] (x,s) dx ds \right],\quad\quad \mathcal{S}_{1}(t)=\mathbb{E}\left[ \int_{0}^{b}\left< \xi_{u}(x,\cdot),\xi_{u}(x,\cdot) \right>_{t}^{T} dx \right],\\
&\mathcal{S}_{2}(t)=-2\mathbb{E} \left[ \int_0^b \int_t^T \xi_u(x,s) d \eta_u(x,s)dx \right], \quad\quad \mathcal{S}_{3}(t)=-2\mathbb{E}\left[ \int_t^T \int_0^b e_\psi(x,s) \xi_u(x,s) dx dW_s \right],\\
  &\mathcal{S}_{4}(t) = \mathbb{E}\left[ \int_t^T \int_0^b \left[ 2\xi_p\eta_v - 2\xi_v\eta_p + \left(\sigma^{2}+\bar{\sigma}^{2}\right) \xi_v \eta_v + 2\sigma \xi_v \eta_\psi - 2\lambda\xi_u\eta_v - 2\mu\xi_u\eta_\psi \right] (x,s) dxds\right],\\
  &\mathcal{S}_{5}(t) = 2\mathbb{E}\left[ \int_t^T \int_0^b \left[ \lambda\xi_v\xi_u + \mu\xi_\psi\xi_u + \left( \Gamma(\cdot,u,v,\psi) - \Gamma(\cdot,u_{h},v_{h},\psi_{h})\right)\xi_u \right] (x,s) dxds\right].
  \end{align*}
  
Now, we estimate the above items in turn.

\textbf{Estimation of $\mathcal{S}_{2}(t)$.}

By the definition \ref{Gauss_Radou_projection_minus} of Gauss-Radau projection $\mathcal{P}^{-}$ and the exchangeability of time variable differentials, from~\eqref{ldg_eq_a} we obtain
\begin{equation}\label{raw_eq_guass_radou}
  -d_{t}\left(\mathcal{P}^{-}u\right)(x,t)=\mathcal{P}^{-}\big(-d_{t} u\big) (x,t)=\mathcal{P}^{-}\big( p_x +\lambda v + \mu \psi + \Gamma(\cdot,u,v,\psi) \big)(x,t) dt -\mathcal{P}^{-}\psi (x,t) dW_{t},
\end{equation}
then $\eta_{u}$ satisfies the following equation 
\begin{equation}\label{error_step1_eta_u_ldg_eq}
  -d\eta_{u}(x,t)=\left(\mathcal{P}^{-}-\mathcal{I}\right)
  \big( p_x +\lambda v + \mu \psi + \Gamma(\cdot,u,v,\psi) \big)(x,t) dt -\left(
  \mathcal{P}^{-} - \mathcal{I}\right)\psi (x,t) dW_{t}.
\end{equation}
Furthermore, for $\mathcal{S}_{2}(t)$, one has
  \begin{align*}
    &-2 \mathbb{E} \left[ \int_{0}^{b} \int_t^T \xi_{u}(x,s)d\eta_{u}(x,s)dx \right] + 2 \mathbb{E} \left[ \int_t^T \int_{0}^{b} \xi_{u}(x,s) \left( \mathcal{P}^{-} -\mathcal{I} \right) \psi(x,s) dx dW_{s} \right]\\
    =& 2 \mathbb{E} \left[ \int_t^T \int_{0}^{b} \xi_{u}(x,s)\left(\mathcal{P}^{-}-\mathcal{I}\right) \big( p_x +\lambda v + \mu \psi + \Gamma(\cdot,u,v,\psi) \big)(x,s) dxds \right].
  \end{align*}  
It is easy to verify that $\int_{0}^{\cdot}\int_{0}^{b}\xi_{u}(x,s)\left(\mathcal{P}^{-}-\mathcal{I}\right)
\left(\psi_{t}(\cdot)\right)(x)dxdW_{t}$ is a martingale. In addition, from Lemma \ref{highordersolution} and the properties satisfied by the Gauss-Radau projection \ref{Gauss_Radou_property}, we deduce that
\begin{equation}\nonumber
  \mathcal{S}_{2}(t)=2 \mathbb{E} \left[ \int_t^T \int_{0}^{b} \left[\xi_{u} \left(\mathcal{P}^{-}-\mathcal{I}\right) \big( p_x +\lambda v + \mu \psi + \Gamma(\cdot,u,v,\psi) \big)\right] (x,s) dxds \right] \leq \mathbb{E}\left[ \int_{t}^{T} \|\xi_{u}(\cdot,s)\|^{2}ds \right]+ C h^{2k+2}.
\end{equation}

\textbf{Estimation of $\mathcal{S}_{1}(t)$.}

By \eqref{raw_eq_guass_radou}, we have for any $r_{h}\in V_{h}^k$,
\begin{equation*} 
-\int_{I_{j}}r_{h}(x)
  d_{t}\left(\mathcal{P}^{-}u\right)(x,t)dx
  +\int_{I_{j}}r_{h}(x)
  \mathcal{P}^{-} \psi(x,t) dxdW_{t}=\int_{I_{j}}r_{h}(x) \mathcal{P}^{-}\big( p_x +\lambda v + \mu \psi + \Gamma(\cdot,u,v,\psi) \big)(x,t)dxdt,
\end{equation*}
Combining \eqref{error_step1_error_def} and \eqref{approx_eq_a} gives
  \begin{align*} 
  &-\int_{I_{j}}r_{h}(x)
  d\xi_{u}(x,t)dx
  +\int_{I_{j}}r_{h}(x)
  \xi_\psi (x,t)dW_{t}\\
  =&\Bigg\{ \int_{I_{j}}p_{h}(x,t)(r_{h})_x(x)dx
        -\hat{p}(x_{j+\frac{1}{2}},t)r_{h}(x_{j+\frac{1}{2}}^{-})
        +\hat{p}(x_{j-\frac{1}{2}},t)r_{h}(x_{j-\frac{1}{2}}^{+}) - \int_{I_{j}}\left[\lambda v_{h} + \mu\psi_{h} + \Gamma(\cdot,u_{h},v_{h},\psi_{h}) \right](x,t) r_{h}(x)dx \\
        & \quad + \int_{I_{j}}r_{h}(x) \mathcal{P}^{-}\big( p_x +\lambda v + \mu \psi + \Gamma(\cdot,u,v,\psi) \big)(x,t)dx \Bigg\}dt.
  \end{align*}
A similar calculation as in deriving the estimate \eqref{entropy_for_uh_estimate} provides
\begin{equation*}
  \mathcal{S}_{1}(t) = \int_0^b \left<\xi_{u}(x,\cdot),\xi_{u}(x,\cdot) \right>_{t}^{T} dx = \mathbb{E} \left[ \int_{t}^{T} \left\|\xi_{\psi}(\cdot,s)\right\|^{2} ds \right].
\end{equation*}

\textbf{Estimation of $\mathcal{S}_{0}(t)$.}

Use the same method as in~\eqref{stability_step1_T_0}, we have
    $$
    \mathcal{S}_{0}(t) \geqslant \frac{\kappa^2}{2} \mathbb{E} \left[ \int_{t}^{T} \|\xi_v(\cdot,s)\|^{2}ds \right] - \left( 1-\frac{\kappa^2}{2K^2+\kappa^2}\right) \mathbb{E} \left[ \int_{t}^{T} \|\xi_\psi(\cdot,s)\|^{2} ds \right].
    $$

\textbf{Estimation of $\mathcal{S}_{3}(t)$.}

We can verify that $\int_0^\cdot \int_0^b e_\psi(x,s) \xi_u(x,s) dx dW_s$ is a martingale, which gives that
$$
\mathcal{S}_{3}(t)\equiv 0.
$$

\textbf{Estimation of $\mathcal{S}_{4}(t)$.}

Taking $z_{h,p}=\xi_p(\cdot,t)$ in the error equation, we get
$$
\mathbb{E} \left[ \int_{t}^{T}\|\xi_p(\cdot,s)\|^{2} ds \right] \leq C \mathbb{E} \left[ \int_{t}^{T}\|\xi_v(\cdot,s)\|^{2} ds \right] + C \mathbb{E} \left[ \int_{t}^{T}\|\xi_\psi(\cdot,s)\|^{2} ds \right] + C h^{2k+2}.
$$
Then with the help of the regularities of $v,p,\psi$ and the boundedness of $\sigma,\bar{\sigma},\lambda,\mu$, we have
\begin{align*}
    \mathcal{S}_{4}(t) &\leq \epsilon \mathbb{E} \left[ \int_{t}^{T} \left( \|\xi_p(\cdot,s)\|^{2} + \|\xi_v(\cdot,s)\|^{2} \right) ds \right] + \mathbb{E} \left[ \int_{t}^{T} \|\xi_u(\cdot,s)\|^{2} ds \right] + C h^{2k+2}\\
    &\leq \epsilon \mathbb{E} \left[ \int_{t}^{T} \left( \|\xi_\psi(\cdot,s)\|^{2} + \|\xi_v(\cdot,s)\|^{2} \right) ds \right] + \mathbb{E} \left[ \int_{t}^{T} \|\xi_u(\cdot,s)\|^{2} ds \right] + C h^{2k+2}.
\end{align*}

\textbf{Estimation of $\mathcal{S}_{5}(t)$.}

It holds that
\begin{align*}
    \mathcal{S}_{5}(t) \leq& \epsilon \mathbb{E} \left[ \int_{t}^{T} \|\xi_v(\cdot,s)\|^{2} ds \right] + \epsilon \mathbb{E} \left[ \int_{t}^{T} \|\xi_\psi(\cdot,s)\|^{2} ds \right] + C \mathbb{E} \left[ \int_{t}^{T} \|\xi_u(\cdot,s)\|^{2} ds \right] + C h^{2k+2}\\
    & + \epsilon \mathbb{E}\left[ \int_t^T \left\| \Gamma(\cdot,s,u,v,\psi) - \Gamma(\cdot,s,u_{h},v_{h},\psi_{h})\right\|^2 ds\right],
\end{align*}
in which
\begin{align*}
\mathbb{E}\left[ \int_t^T \left\| \Gamma(\cdot,s,u,v,\psi) - \Gamma(\cdot,s,u_{h},v_{h},\psi_{h})\right\|^2 ds\right] \leq& \mathbb{E}\left[ \int_t^T \left\| e_u(\cdot,s) \right\|^2 + \left\| e_v(\cdot,s) \right\|^2 + \left\| e_\psi(\cdot,s) \right\|^2 ds\right]\\
\leq& \mathbb{E}\left[ \int_t^T \left\| \xi_u(\cdot,s) \right\|^2 + \left\| \xi_v(\cdot,s) \right\|^2 + \left\| \xi_\psi(\cdot,s) \right\|^2 ds\right] + C h^{2k+2}.
\end{align*}

Concluding the above estimates, for fixed $\epsilon$ small enough, we obtain
\begin{align*}
    &\mathbb{E}\left[\|\xi_u(\cdot,s)\|^2\right] + \mathbb{E}\left[ \int_{t}^{T}  \|\xi_v(\cdot,s)\|^2 ds\right] + \mathbb{E}\left[ \int_{t}^{T}  \|\xi_\psi(\cdot,s)\|^2 ds\right] \leq C \mathbb{E}\left[ \int_{t}^{T} \|\xi_u(\cdot,s)\|^2 ds\right] + C h^{2k+2}.
\end{align*}
By Gronwall's inequality, we derive that 
\begin{equation*}
  \sup_{0\leq t \leq T} \mathbb{E}\left[\|\xi_{u}(\cdot,t)\|^{2}\right]
  \leq Ch^{2k+2},
\end{equation*}
which yields that
\begin{equation*}
  \mathbb{E}
  \Big[\int_{0}^{T}\|\xi_{v}(x,s)\|^{2}ds\Big]
  \leq Ch^{2k+2}\quad \mbox{and}\quad\mathbb{E}
  \Big[\int_{0}^{T}\|\xi_{\psi}(x,s)\|^{2}ds\Big]
  \leq Ch^{2k+2}.
\end{equation*}

On the other hand, as $u\in\mathcal{S}_{\mathscr{F}}^{2}(0,T;H^{k+1,2})$, $v,\psi\in\mathcal{L}_{\mathscr{F}}^{2}(0,T;H^{k+1,2})$, one has
\begin{align}
  &\mathbb{E}\left[ \sup_{0\leq t \leq T} \left\|\eta_{u}(\cdot,t)\right\|^{2} + \int_{0}^{T} \left\| \eta_{v}(\cdot,s)\right\|^{2} ds + \int_0^T \left\| \eta_{\psi}(\cdot,s)\right\|^{2} ds\right]\nonumber\\
  \leq& Ch^{2k+2} \mathbb{E}\left[ \sup_{0\leq t \leq T} \left\| u(\cdot,t) \right\|^{2}_{H^{k+1}} + \int_{0}^{T} \left\| v(\cdot,t) \right\|^{2}_{H^{k+1}} ds + \int_{0}^{T} \left\| \psi(\cdot,t) \right\|^{2}_{H^{k+1}} ds \right] \leq Ch^{2k+2},\label{eq202408081}
\end{align}
then a simple calculation provides
\begin{equation}\label{error_step1_conclusion}
  \begin{aligned}
  \mathop{\text{sup}}_{0\leq t\leq T}\mathbb{E}\left[
  \left\|u(\cdot,t)-u_{h}(\cdot,t)\right\|^{2}
  \right]+
  \mathbb{E}\left[\int_{0}^{T}
  \left\|u_x(\cdot,s)-v_{h}(\cdot,s)\right\|^{2}ds\right]
  +\mathbb{E}\left[\int_{0}^{T}
  \left\|\psi(\cdot,s)-\psi_{h}(\cdot,s)\right\|^{2}ds\right] \leq Ch^{2k+2}.
  \end{aligned}
\end{equation}

Next, based on above results, we are going to estimate $\mathbb{E}\left[ \mathop{\text{sup}}_{0\leq t\leq T} \left\|u(\cdot,t)-u_{h}(\cdot,t)\right\|^{2} \right]$. According to~\eqref{error_step1_error_equation_plugin_0}, we get
\begin{align*}
    \mathbb{E}\left[ \sup_{r\in[t,T]} \|\xi_u(\cdot,r)\|^2 \right] \leq \mathbb{E}\left[ \|\xi_u(\cdot,T)\|^2 \right] + \mathcal{S}_{6}(t) + \mathcal{S}_{7}(t) + \mathcal{S}_{8}(t),
\end{align*}
with
  \begin{align*}
  \mathcal{S}_{6}(t)&= \!2 \mathbb{E}\left[ \int_t^T\!\!\!\! \int_0^b \left| \eta_v\xi_p - \eta_p\xi_v + \!\frac{\sigma^{2}+\bar{\sigma}^{2}}{2} \eta_v \xi_v +\sigma \eta_\psi \xi_v + \lambda e_v \xi_u + \mu e_\psi \xi_u + \left(\Gamma(\cdot,u,v,\psi) - \Gamma(\cdot,u_{h},v_{h},\psi_{h})\right)\xi_u \right| \!(x,s) dx ds \!\right]\!,\\
  \mathcal{S}_{7}(t)&= 2 \mathbb{E}\left[ \sup_{r\in[t,T]} \left| \int_0^b \int_r^T \xi_u(x,s) d \eta_u(x,s)dx \right| \right],\\
  \mathcal{S}_{8}(t)&= 2 \mathbb{E}\left[ \sup_{r\in[t,T]} \left| \int_r^T \int_0^b e_\psi(x,s) \xi_u(x,s) dx dW_s \right| \right].
  \end{align*}
  
We now estimate the above items in turn.

\textbf{Estimation of $\mathcal{S}_{6}(t)$.}

In view the reults in first step, it follows that
\begin{align*}
\mathcal{S}_{6}(t) \leq & C \mathbb{E}\left[ \int_t^T \left\|\eta_v(\cdot,s)\right\|^2 + \left\|\xi_p(\cdot,s)\right\|^2 + \left\|\eta_p (\cdot,s)\right\|^2 + \left\|\xi_v(\cdot,s)\right\|^2 + \left\|\eta_\psi(\cdot,s)\right\|^2  ds \right]\\
& +C \mathbb{E}\left[ \int_t^T \left\|e_v(\cdot,s)\right\|^2 + \left\|\xi_u(\cdot,s)\right\|^2 + \left\|e_\psi(\cdot,s)\right\|^2 + \left\|e_u(\cdot,s)\right\|^2 ds \right] \leq Ch^{2k+2}.
\end{align*}

\textbf{Estimation of $\mathcal{S}_{7}(t)$.}

From~\eqref{error_step1_eta_u_ldg_eq} and BDG inequality, one has
\begin{align*}
    \mathcal{S}_{7}(t) \leq & 2 \mathbb{E} \left[ \sup_{r \in[t,T]} \left| \int_r^T \int_{0}^{b} \xi_{u}(x,s) \left( \mathcal{P}^{-} -\mathcal{I} \right) \psi(x,s) dx dW_{s} \right| \right]\\
    & + 2 \mathbb{E} \left[ \int_t^T \int_{0}^{b} \left| \xi_{u}(x,s)\left(\mathcal{P}^{-}-\mathcal{I}\right) \big( p_x +\lambda v + \mu \psi + \Gamma(\cdot,u,v,\psi) \big)(x,s) \right| dxds \right]\\
    \leq & C \mathbb{E} \left[ \left( \int_t^T \left( \int_{0}^{b} \xi_{u}(x,s) \left( \mathcal{P}^{-} -\mathcal{I} \right) \psi(x,s) dx \right)^2 ds \right)^{\frac{1}{2}} \right]\\
    &+ \mathbb{E} \left[ \int_t^T \left\| \xi_{u}(\cdot,s) \right\|^2 ds \right] + \mathbb{E} \left[ \int_t^T \left\| \left(\mathcal{P}^{-}-\mathcal{I}\right) \big( p_x +\lambda v + \mu \psi + \Gamma(\cdot,u,v,\psi) \big) (\cdot,s) \right\|^2 ds \right]\\
    \leq & C \mathbb{E} \left[ \sup_{s\in[t,T]} \left\| \xi_{u}(\cdot,s) \right\| \left( \int_t^T \left\| \left( \mathcal{P}^{-} -\mathcal{I} \right) \psi(\cdot,s) \right\|^2 ds \right)^{\frac{1}{2}} \right] + \mathbb{E} \left[ \int_t^T \sup_{r\in[s,T]}\left\| \xi_{u}(\cdot,r) \right\|^2 ds \right] + Ch^{2k+2}\\
    \leq & \frac{1}{3}\mathbb{E}\left[ \mathop{\text{sup}}_{r\in[t,T]} \left\|\xi_u(\cdot,r)\right\|^{2} \right] +  \int_t^T \mathbb{E} \left[ \sup_{r\in[s,T]}\left\| \xi_{u}(\cdot,r) \right\|^2 \right] ds + Ch^{2k+2}.
  \end{align*}  
  
\textbf{Estimation of $\mathcal{S}_{8}(t)$.}

Using BDG inequality, we have
\begin{align*}
    \mathcal{S}_{8}(t) \leq & C \mathbb{E} \left[ \left( \int_t^T \left( \int_{0}^{b} e_\psi(x,s) \xi_{u}(x,s) dx \right)^2 ds \right)^{\frac{1}{2}} \right] \leq \frac{1}{3}\mathbb{E}\left[ \mathop{\text{sup}}_{r\in[t,T]} \left\|\xi_u(\cdot,r)\right\|^{2} \right] + Ch^{2k+2}.
\end{align*}

Concluding the above estimates, we get
\begin{equation*}
  \mathbb{E}\left[ \mathop{\text{sup}}_{r\in[t,T]} \left\|\xi_u(\cdot,r)\right\|^{2} \right] \leq C \int_t^T \mathbb{E} \left[ \sup_{r\in[s,T]}\left\| \xi_{u}(\cdot,r) \right\|^2 \right] ds + Ch^{2k+2},
\end{equation*}
then thanks to Gronwall's inequality, it holds
\begin{equation}\nonumber
  \mathbb{E}\left[
  \mathop{\text{sup}}_{0 \leq r\leq T}
  \|\xi_{u}(\cdot,r)\|^{2}\right]
  \leq Ch^{2k+2}.
\end{equation}
Combining with~\eqref{eq202408081} and~\eqref{error_step1_conclusion}, we have
\begin{equation}\nonumber
    \Big(\mathbb{E}\Big[\mathop{\text{sup}}_{0\leq t\leq T}
    \|u(\cdot,t)-u_{h}(\cdot,t)\|^{2}\Big]
    \Big)^{\frac{1}{2}}
    +\Big(\mathbb{E}\Big[\int_{0}^{T}
    \|u_{x}(\cdot,s)-v_{h}(\cdot,s)\|^{2}ds\Big]
    \Big)^{\frac{1}{2}}
    +\Big(\mathbb{E}\Big[\int_{0}^{T}
    \|\psi(\cdot,s)-\psi_{h}(\cdot,s)\|^{2}ds\Big]
    \Big)^{\frac{1}{2}}
    \leq Ch^{k+1}.
\end{equation}

\section{Numerical Experiments}

\subsection{Deep backward dynamic programming algorithm}
After applying the LDG spatial discretization~\eqref{approx_eq_a}-\eqref{approx_eq_d}, we need to solve the high dimensional nonlinear BSDE~\eqref{approx_bsde}. To address the curse of dimensionality associated with using a Monte Carlo scheme for high-dimensional nonlinear BSDEs, we will employ the deep backward dynamic programming algorithm. This algorithm was first introduced in \cite{hure2020deep} and later applied to solve coupled forward-backward stochastic differential equations (FBSDEs) in \cite{molla2023numerical}. Note that BSDE \eqref{approx_bsde} can be formally regarded as an FBSDE with a forward equation that is a one-dimensional Brownian motion. Then the solution $\big(\boldsymbol{u},\boldsymbol{\psi}\big)$ and the terminal condition $\boldsymbol{g}$ of \eqref{approx_bsde} can be regarded as the functionals of Brownian motion $W$, i.e.,
\begin{equation*}
    \boldsymbol{u}_{t}=\boldsymbol{U}(t,W_{t}),\quad
    \boldsymbol{\psi}_{t}=\boldsymbol{\Psi}(t,W_{t}),\quad
    \boldsymbol{g}=\boldsymbol{G}(W_{T}).
\end{equation*}

To give our algorithm, we consider the time division on $[0,T]$ with $0=t_{0}<t_{1}<\cdots<t_{M}=T$ and $\Delta t_{i}=t_{i+1}-t_{i}$. Then for $i=0,1,...,M-1$, we define $X_{t_{i+1}}=X_{t_{i}}+\Delta W_{t_{i}}$ with $X_{t_{0}}=0$ and $W_{t_{i}}=W_{t_{i+1}}-W_{t_{i}}$. Use $\boldsymbol{\hat{U}}, \boldsymbol{\hat{\Psi}}$ to represent the approximate estimates of the functions $\boldsymbol{U}, \boldsymbol{\Psi}$ respectively.
The following are the simple steps of the algorithm:

    \textbf{Step1.} Initialize the estimation term $\boldsymbol{\hat{U}}_{M}(W_{T})$ according to the terminal term $\boldsymbol{G}(W_{T})$.

    \textbf{Step2.} For $i=M-1,...,0$, use a set of neural networks $\big(U_{i}(\cdot\,;\theta_{U}),\Psi_{i}(\cdot\,;\theta_{\Psi})\big)$ to approximate 
     $\big(\boldsymbol{U}_{{i}},\boldsymbol{\Psi}_{{i}}\big)$ such that
     the optimal parameters are obtained by minimizing the following loss function
    \begin{equation}\label{algo_on_nerual_network}
        \left\{
        \begin{aligned}
        &L_{i}(\theta_{U},\theta_{\psi}):=
        \mathbb{E}\Big[\big|
        \boldsymbol{\hat{U}}_{{i+1}}(X_{t_{i+1}})
        -\big(U_{i}(X_{t_{i}};\theta_{U}) -F_{h}\big(t_{i},U_{i}(X_{t_{i}};\theta_{U}),
        \Psi_{i}(X_{t_{i}};\theta_{\Psi})\big)\Delta t_{i}
        +\Psi_{i}(X_{t_{i}};\theta_{\Psi})\Delta W_{t_{i}}\big)
        \big|^{2}\Big],\\
        &(\theta_{U}^{i,\ast},\theta_{\psi}^{i,\ast})\in
        \mathop{\text{argmin}}_{\theta_{U},\theta_{\psi}}
        L_{i}(\theta_{U},\theta_{\psi}).
        \end{aligned}\right.
    \end{equation}
The above algorithm uses two neural networks with independent parameters $\left(U_{i}(\cdot;\theta_{U}),\Psi_{i}(\cdot;\theta_{\Psi})\right)$ and the division on the given time interval to solve the numerical approximation solution backwardly and stepwise.
The single-step loss function is obtained through the following approximation
\begin{equation*}
    u_{t_{i+1}} \approx u_{t_{i}}
    -F_{h}\big(t_{i},u_{t_{i}},\psi_{t_{i}}\big)\Delta t_{i}
    +\psi_{t_{i}}\Delta W_{t_{i}},
\end{equation*}
Finally, $U_{0}(X_{0};\theta_{U}^{0,\ast})$ is the numerical solution of BSDE \eqref{approx_bsde} at $t=0$, which is brought back to the LDG method framework to obtain the numerical solution to equation \eqref{BSPDE_in_Qiu_paper} at time zero.

\begin{remark}
    Since exploring network structures is not the focus of this article, we use a relatively simple network for the numerical experiments, consisting of two fully connected hidden layers and four batch normalization layers. The dimension of the hidden layer is that of the solution space of \eqref{approx_bsde} plus $10$, and the dimension of the output layer is the same as that of the solution space. For $i=0,1,...,M-1$, the parameters $\big(U_{i}(\cdot;\theta_{U}),\Psi_{i}(\cdot;\theta_{\Psi})\big)$ of the neural networks at each time are independent of each other.
\end{remark}

In the algorithm \ref{algo_on_nerual_network} all the sub-networks are four-layer fully connected neural networks, which are $(k+1)\cdot N$-dimensional input layers and two layers of $(k+1)\cdot N +10 $-dimensional hidden layer and $(k+1)\cdot N$-dimensional output layer. Use ReLU as the activation function, and perform batch normalization after each linear transformation and before the activation function.

We use TensorFlow to implement the algorithm \ref{algo_on_nerual_network} and use the Adam optimizer to solve for optimal parameters. All numerical experiments are run on a Legion Y9000P equipped with a 2.2-GHz Intel i9 Core processor, NVIDIA RTX4060 graphics card, and 16GB of memory.

\subsection{Numerical Examples}

\subsubsection{BSPDEs with globally Lipschitz continuous coefficients}

We first consider the nonlinear BSPDEs~\eqref{BSPDE_in_Qiu_paper} on $[0,2\pi]$ with following coefficients
\begin{align}
&\Gamma(x,t,u,v,p)= 
\ln\left(1+e^{v}\right) +\frac{1}{2}\sin^{2}x\cos xe^{-\frac{1}{2}(T-t)}\cos W_{t} - \sin^{2}xe^{-\frac{1}{2}(T-t)}\sin W_{t} - \ln\left(1+e^{-\sin xe^{-\frac{1}{2}(T-t)}\cos W_{t}}\right),\nonumber\\
&\sigma(x,t)=\sin x,\quad\quad \bar{\sigma}(t,x)=1,\quad\quad g(x,t)=0,\quad\quad  G(x)=\cos x\cos W_{T},\label{bspde_ex_1}
\end{align}
which satisfies Lipschitz assumption $(\mathcal{A}_3)$. It is easy to verify that the analytical solution of equation with coefficients~\eqref{bspde_ex_1} is as follows
\begin{equation}\label{exactbspde}
    u(x,t)=\cos xe^{-\frac{1}{2}(T-t)}\cos W_{t},\quad
    \psi(x,t)=-\cos xe^{-\frac{1}{2}(T-t)}\sin W_{t}.
\end{equation}

\begin{figure}[htbp]
\centering
\includegraphics[width=0.49\textwidth]{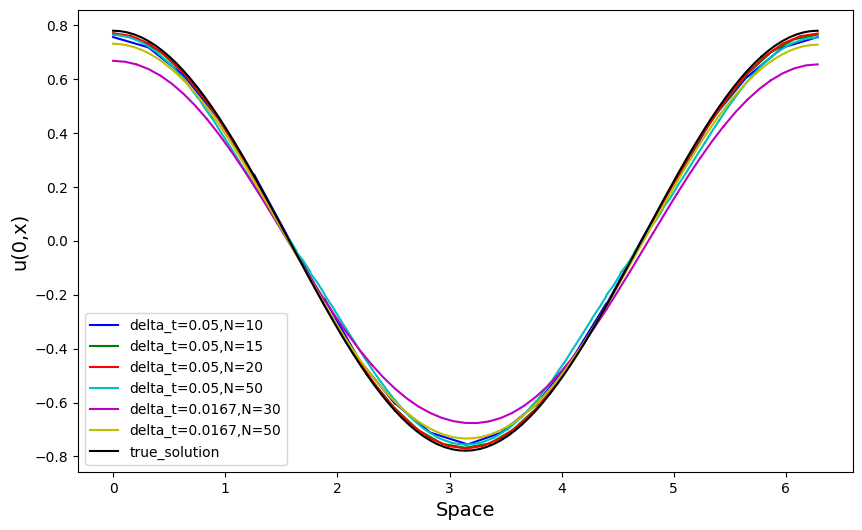}
\includegraphics[width=0.49\textwidth]{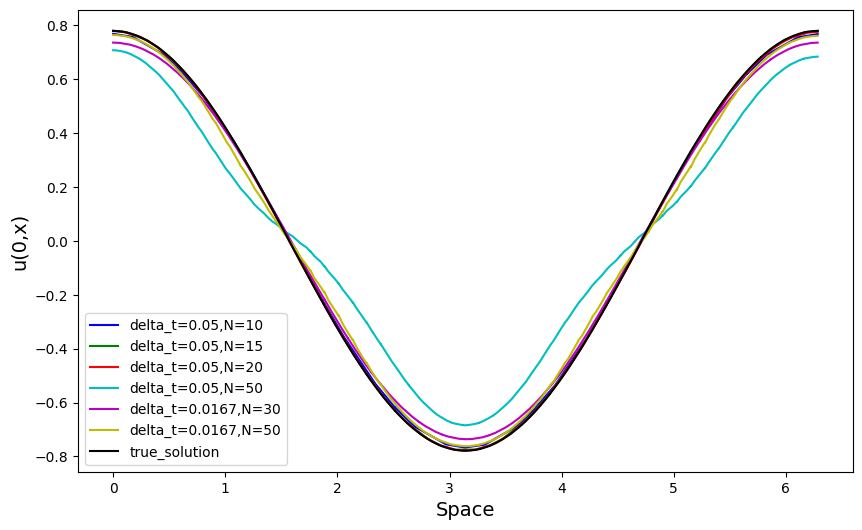}
\caption{Accuracy on \eqref{bspde_ex_1}: $T=0.5$, $k=2$ (left), $k=3$ (right).}
\label{figure1}
\end{figure}

\begin{table}[htbp]
    \centering
    \begin{tabularx}{0.6\textwidth}{|X|X|X|X|}
    \hline
    $N$ & $\Delta_t$ & ${R_E,k = 2}$ & ${R_E,k = 3}$ \\ \hline
    10 &   0.0500 &  0.001114 &  0.000219 \\ \hline
    15 &   0.0500 &  0.000220 &  0.000005 \\ \hline
    20 &   0.0500 &  0.000150 &  0.000025 \\ \hline
    50 &   0.0500 &  0.003989 &  0.053577 \\ \hline
    30 &   0.0167 &  0.019569 &  0.002827 \\ \hline
    50 &   0.0167 &  0.003312 &  0.003200 \\ \hline
    \end{tabularx}
    \caption{Accuracy on \eqref{bspde_ex_1}}
    \label{table1}
\end{table}

On each finite element interval $I_{j}$, let $k=2,3$, that is, choose 2 or 3 degree Lagrange polynomials as the basis. By using deep backward dynamic programming algorithm \eqref{algo_on_nerual_network}, setting
$\Delta t=0.05$ with $N=10,15,20,50$ and $\Delta t=0.0167$ with $N=30,50$, we solve the numerical solution for \eqref{bspde_ex_1} respectively, where $\Delta t$ is the uniform step size of time, $N$ is the number of space divisions, and the solution space dimension is $(k+1)\cdot N$. Use the previous notation, $u_{h}(x, 0)$ represents the numerical solution of $u_{0}(x)$. We consider the following relative error to examine the approximation effect 
\begin{equation*}
    R_{E}=\frac{\int_{0}^{2\pi}
    \left|u_{h}(x,0)-u_{0}(x)\right|^{2}dx}
    {\int_{0}^{2\pi}
    \left|u_{0}(x)\right|^{2}dx}.
\end{equation*}

According to Table~\ref{table1}, as the spatial mesh is smaller, the relative error decreases rapidly. When $\Delta_{t}=0.05$ and the number of divisions $N$ is moderate, the rate of decrease in relative error for $k=3$ is higher than for $k=3$. This is consistent with the conclusion of Theorem~\ref{theorem_ldg_optimal_error}. However, when the space division is too small, such as $\Delta_t=0.05$ with $N=50$, the convergence accuracy is not as good as when the number of divisions is moderate.

\subsubsection{BSPDEs with polynomial growing coefficients}

Now we consider an example in which the globally Lipschitz condition $(\mathcal{A}_{3})$ is not satisfied,
\begin{align}
&\Gamma(x,t,u,v,p)= 
u^2 + \frac{1}{2}\sin^{2}x\cos xe^{-\frac{1}{2}(T-t)}\cos W_{t} - \sin^{2}xe^{-\frac{1}{2}(T-t)}\sin W_{t} - \cos^{2}xe^{-(T-t)}\cos^{2}W_{t},\nonumber\\
&\sigma(x,t)=\sin x,\quad\quad \bar{\sigma}(t,x)=1,\quad\quad g(x,t)=0,\quad\quad  G(x)=\cos x\cos W_{T}.\label{bspde_ex_2}
\end{align}
It is easy to verify the analytical solution of equation with the coefficients~\eqref{bspde_ex_2} is also~\eqref{exactbspde}. 

\begin{figure}[htbp]
\centering
\includegraphics[width=0.49\textwidth]{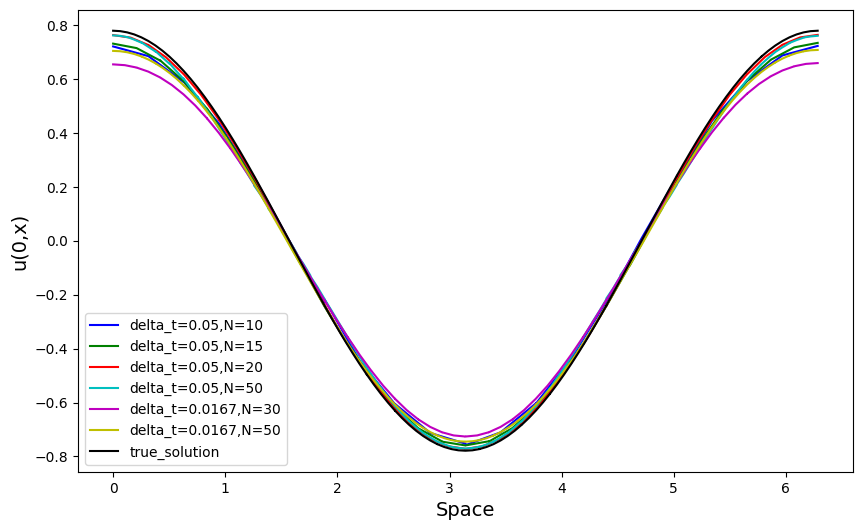}
\includegraphics[width=0.49\textwidth]{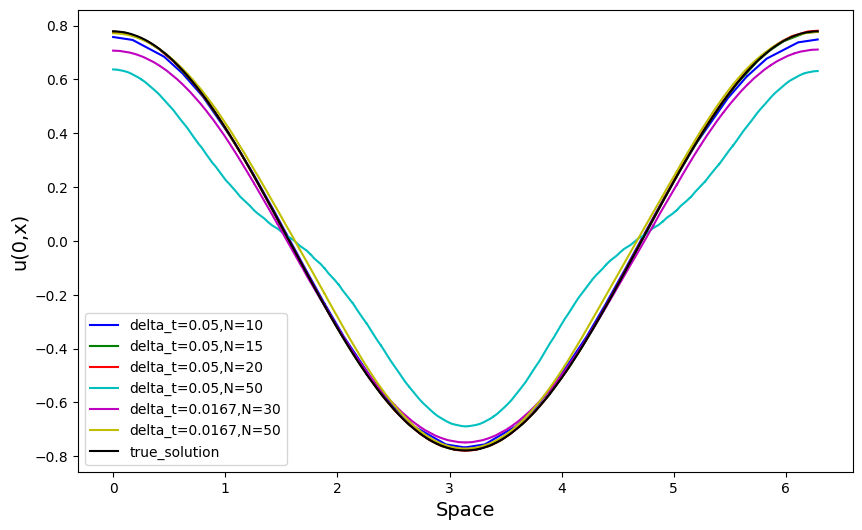}
\caption{Accuracy on \eqref{bspde_ex_2}: $T=0.5$, $k=2$ (left), $k=3$ (right).}
\label{figure2}
\end{figure}

\begin{table}[htbp]
\centering
\begin{tabularx}{0.6\textwidth}{|X|X|X|X|}
\hline
$N$ & $\Delta_t$ & ${R_E,k = 2}$ & ${R_E,k = 3}$ \\ \hline
10 &   0.0500 &  0.002950 &  0.000534 \\ \hline
15 &   0.0500 &  0.002146 &  0.000004 \\ \hline
20 &   0.0500 &  0.000237 &  0.000005 \\ \hline
50 &   0.0500 &  0.002055 &  0.077633 \\ \hline
30 &   0.0167 &  0.012780 &  0.004351 \\ \hline
50 &   0.0167 &  0.004741 &  0.001719 \\ \hline
\end{tabularx}
\caption{Accuracy on \eqref{bspde_ex_2}.}
\label{table2}
\end{table}

According to Figure~\ref{figure1} and Table~\ref{table2}, we observe that even if the nonlinear term of equation \eqref{bspde_ex_2} does not satisfy the globally Lipschitz condition, the LDG method with the deep backward dynamic programming time-marching still gives higher convergence accuracy. 

\begin{remark}
As the time mesh $\Delta_{t}$ decreases, algorithm \ref{algo_on_nerual_network} necessitates solving an increased number of sub-optimization problems, which raises the computational cost. Therefore, the time mesh should not be too small. Additionally, due to the instability in the optimal parameter solutions of Algorithm \ref{algo_on_nerual_network}, verifying the convergence order outlined in Theorem \ref{theorem_ldg_optimal_error} is challenging in this numerical experiment. Further exploration is needed to achieve high-order precision numerical solutions for nonlinear BSPDEs.
\end{remark}


\bibliographystyle{abbrv}
\bibliography{paper}

\end{document}